\documentclass[12pt]{amsart}

\usepackage{bbm}
\usepackage{amssymb}
\usepackage{amsmath}
\usepackage[all]{xy}
\usepackage{tikz}
\usepackage{amsfonts}
\usepackage{mathrsfs}
\usepackage{latexsym}
\usepackage{graphicx}
\usepackage{enumerate}
\usepackage{amscd,amssymb,amsmath,amsbsy,amsthm}
	\definecolor{linkred}{rgb}{0.7,0.2,0.2}
	\definecolor{linkblue}{rgb}{0,0.2,0.6}
	\definecolor{linkgreen}{rgb}{0,0.6,0.2}
\usepackage[colorlinks,plainpages,backref,
    linkcolor=linkblue,
    citecolor=linkgreen,
    urlcolor=linkred]{hyperref}

\newtheorem{Th}{Theorem}[section]
\newtheorem{Lemma}[Th]{Lemma}
\newtheorem{Defin}[Th]{Definition}
\newtheorem{Coro}[Th]{Corollary}
\newtheorem{Prop}[Th]{Proposition}
\newtheorem{Eg}[Th]{Example}
\newtheorem{Rmk}[Th]{Remark}

\allowdisplaybreaks

\newtheorem{Thm}{Theorem}[section]

\def\CSM{c_{\rm SM}}
\def\SSM{s_{\rm SM}}
\def\Gr{\operatorname{Gr}}

\def\midarrow{\tikz{\draw[thick,->] (-0.001,0) -- +(0.001,0);}}
\def\midlabel#1{\tikz{%
    \draw[fill=white,white] (0,0) circle [radius=.15];
    \node at (0,0){\small #1};}}
\def\wt{\operatorname{wt}}

\newcommand{\PD}[2][1pc]{%
\setlength{\unitlength}{#1}
\def\BPDframe{%
    \thinlines%
    \color{lightgray}%
    \put(0,0){\line(0,1){1}}%
    \put(1,0){\line(0,1){1}}%
    \put(0,0){\line(1,0){1}}%
    \put(0,1){\line(1,0){1}}%
    \linethickness{0.08\unitlength}
    \color{teal}}
\def\O{
\begin{picture}(1,1)
    \BPDframe
\end{picture}}
\def\X{
\begin{picture}(1,1)
    \BPDframe
    \qbezier(0.5,0)(0.5,0.5)(0.5,1)
    \qbezier(0,0.5)(0.5,0.5)(1,0.5)    
\end{picture}}
\def\B{
\begin{picture}(1,1)
    \BPDframe
    \qbezier(0.5,0)(0.5,0.5)(1,0.5)
    \qbezier(0.5,1)(0.5,0.5)(0,0.5)
\end{picture}}
\def\b{
\begin{picture}(1,1)
    \BPDframe
    \qbezier(0.5,0)(0.5,0.5)(0,0.5)
    \qbezier(0.5,1)(0.5,0.5)(1,0.5)
\end{picture}}
\def\D{
\begin{picture}(0,1)
    \color{green}
    \linethickness{0.1\unitlength}
    \qbezier(0,0)(0,0.1)(0,0.1)
    \qbezier(0,0.4)(0,0.6)(0,0.6)
    \qbezier(0,0.9)(0,1)(0,1)
\end{picture}
}
\def\M##1{\begin{picture}(1,1)%
    \put(0,0.2){\makebox[\unitlength]{\(##1\)}}
\end{picture}}
\begin{array}{@{\,}c@{\,}}
{\def\arraystretch{0}
\setlength{\arraycolsep}{0pc}
\color{teal}
\begin{array}{@{}l@{}}%
#2\end{array}}
\end{array}}

\def\T{T}
\def\id{\operatorname{id}}
\def\pt{\operatorname{pt}}
\def\Gr{\operatorname{Gr}}
\def\Fl{\operatorname{Fl}}
\def\Hom{\operatorname{Hom}}
\def\Sym{\operatorname{Sym}}
\def\aff{\operatorname{aff}}
\def\loc{\operatorname{loc}}
\def\rot{\operatorname{rot}}
\def\Ad{\operatorname{Ad}}
\def\ev{\operatorname{ev}}
\def\t{g}

\begin{document}

\title[Chern classes of open Projected Richardsons and of affine Schuberts]{Chern Classes of Open Projected Richardson Varieties and of Affine Schubert Cells}

\author{Neil J.Y. Fan}
\address[Neil J.Y. Fan]{Department of Mathematics, 
Sichuan University, Chengdu, Sichuan 610065, P.R. China}
\email{fan@scu.edu.cn}

\author{Peter L. Guo}
\address[Peter L. Guo]{Center for Combinatorics, LPMC, 
Nankai University, Tianjin 300071, P.R. China}
\email{lguo@nankai.edu.cn}

\author{Changjian Su}
\address[Changjian Su]{Yau Mathematical Sciences Center, Tsinghua University, Beijing, China}
\email{changjiansu@mail.tsinghua.edu.cn}

\author{Rui Xiong}
\address[Rui Xiong]{Department of Mathematics and Statistics, University of Ottawa, 150 Louis-Pasteur, Ottawa, ON, K1N 6N5, Canada}
\email{rxion043@uottawa.ca}


\begin{abstract}
    The open projected Richardson varieties form a stratification for the partial flag variety $G/P$. We compare the Segre--MacPherson classes of open projected Richardson varieties with those of the corresponding affine Schubert cells by pushing or pulling them to the affine Grassmannian. In the  Grassmannian case, the open projected Richardson varieties are well known as open positroid varieties. We obtain symmetric functions that represent the Segre--MacPherson classes of open positroid varieties, constructed explicitly in terms of pipe dreams for affine permutations.
\end{abstract}

\maketitle

\section{Introduction}

Let $G$ be a reductive group such that the derived subgroup $G'$ is simple, and let $B$ and $B^-$ be the Borel and opposite Borel subgroups,   $T=B\cap B^-$  the maximal torus, and $W$ the associated Weyl group. 
For   $u\leq w\in W$ in the Bruhat order, the {\it open Richardson variety} $\mathring{R}_{u,w}$ over the full flag variety $G/B$ is the intersection of the Schubert cell $\mathring{\Sigma}_{w}=BwB/B$ and the opposite Schubert cell $\mathring{\Sigma}^{u} = B^-uB/B$, whose closure is the {\it closed Richardson variety} $R_{u,w}$. 

Fix a parabolic subgroup $P$ containing $B$. Let $\pi\colon G/B\rightarrow G/P$ be the natural projection. The {\it open projected Richardson variety} is  $\mathring{\Pi}_{u,w}:=\pi(\mathring{R}_{u,w})$. Its closure $\Pi_{u,w}:=\pi(R_{u,w})$ is the {\it closed projected  Richardson variety}, which originates from the study of total positivity and Poisson geometry, see for example \cite{Lus98,Rie06,GY09}. 
Let $W_P$ be the Weyl group of $P$, and $W^P$ the   minimal length coset representatives of $W/W_P$. Knutson, Lam and Speyer \cite{KLS-Proj} showed that the open projected Richardson varieties $\mathring{\Pi}_{u,w}$, where $w$ ranges over elements in  $W^P$ (or equivalently, over equivalence classes of $P$-Bruhat intervals \cite[Section 2]{KLS-Proj}), form a stratification of $G/P$. 
Many of the geometric properties of Richardson varieties were shown to hold for projected Richardson varieties  \cite{KLS-Proj}, see also Billey and Coskun \cite{BC}. 

The projected Richardson varieties over Grassmannians are   known as \emph{positroid varieties} studied systematically by Knutson, Lam and Speyer \cite{KLS-Posi},  motivated by previous work of   Postnikov \cite{P06}. 
Positroid varieties have drawn  increasing attention  in combinatorics, representation theory, and algebraic geometry. For example, they are related to affine Grassmannian \cite{HL15},
Gromov--Witten invariants \cite{BKT03,BCMP18}, cluster algebras \cite{GL19}, knot invariants \cite{GL24}, and we refer the readers to the  excellent surveys \cite{Lam14, S23}. 

In this paper, we  investigate  the Chern--Schwartz--MacPherson (CSM) and Segre--MacPherson (SM) classes of  open projected Richardson varieties. 
These classes are  generalizations of the usual Chern classes of smooth varieties to singular varieties $X$, see for example \cite{MacPherson, Sch65a, Sch65b, Ohmoto}. 
They are assigned to constructible functions on the variety, and behave well under pushforwards. 
We will focus on the characteristic function $\mathbbm{1}_Y$ of certain locally closed subvariety $Y$ inside $X$. 
When $X=G/B$, the CSM classes of   Schubert cells $\mathring{\Sigma}_w$ are equivalent to the Maulik--Okounkov stable envelopes, thereby having close connections with the representation theory of the group $G$, see \cite{AMSS23, MO19}. 

On the other hand, to the group $G$, we have the affine flag variety $\Fl_G$ and the affine Grassmannian $\Gr_G$, which are both infinite dimensional variants of the finite flag variety. They play a crucial  role in the geometric representation theory \cite{Zhu}, and draw special interests because of  their relation to quantum Schubert calculus \cite{P97,LS10,Kato}. 
For an element $f$ in the  extended affine Weyl group $\widehat{W}$, the {\it affine Schubert cell} $\mathring{\Sigma}_f\subset \Fl_G$ is finite-dimensional. 
So its CSM class $\CSM(\mathring{\Sigma}_f)$ is well defined in the (small) torus equivariant homology group $H^T_*(\Fl_G)$. 
These classes form a basis for the localized equivariant homology group. Inspired by the behavior of CSM and SM classes over the finite type flag varieties \cite{AMSS23}, we define the SM classes $\SSM(\mathring{\Sigma}^f)$ of the {\it opposite affine Schubert cells} $\mathring{\Sigma}^f$ to be the dual basis of the CSM classes, similar to the Schubert classes considered  by Kostant and Kumar \cite{KK86}.

Recall the affine Grassmannian $\Gr_G=G(\!(z)\!)/G[[z]]$. 
Let us fix a dominant cocharacter $\lambda$, such that the stabilizer subgroup $W_\lambda$ equals the parabolic subgroup  $W_P$. 
Then the $G$-orbit of the element $z^{-\lambda}G[[z]]/G[[z]]\in \Gr_G$ is isomorphic to $G/P$, and its $G[[z]]$-orbit $\Gr_\lambda$, the spherical Schubert variety, is an affine bundle over the partial flag variety $G/P$. 
Let us consider the following diagram,
$$\xymatrix{
G/P\ar[r]^-{i_\lambda} & \Gr_\lambda\ar[r]^-{j_\lambda}& \Gr_G \ar[r]^-r &  \Fl_G,}$$
where $i_\lambda$ and $j_\lambda$ are  inclusions, and $r$ is a continuous section of the projection $r: \Fl_G\to \Gr_G$ defined as follows. 
Let $K\subset G$ be the maximal compact subgroup and $T_\mathbb{R}:=K\cap T$  the compact torus. Then $r$ is the continuous map 
\[r\colon \Gr_G\simeq \Omega K\rightarrow LK\rightarrow LK/T_\mathbb{R}\simeq \Fl_G,\] 
where $LK$ (resp., $\Omega K$) is the space of polynomial maps $S^1\to K$ (resp. $(S^1,1)\to (K,1)$). Therefore, we can pushforward $i_{\lambda,*}$ classes on $G/P$ to $\Gr_\lambda$, and also pullback $j_\lambda^*\circ r^*$ classes on $\Fl_G$ to $\Gr_\lambda$.

Our first main result establishes a relationship between the SM classes of open projected Richardson varieties and the SM classes of opposite affine Schubert cells in the affine flag variety. 

\begin{Thm}[Theorem \ref{thm:pR=aff}]\label{thm-B}
Let $\mathcal{N}$ be the normal bundle of $G/P$ inside $\Gr_\lambda$. Then, for  $u\leq w$ with $w\in W^P$,
$$i_{\lambda,*}\bigg(\SSM(\mathring{\Pi}_{u,w})\cdot c^T(\mathcal{N})^{-1}\bigg)
= (j_\lambda^*\circ r^*)\bigg(\SSM(\mathring{\Sigma}^{f})\bigg)
\in H_T^*(\Gr_\lambda)_{\loc},
$$
where $f=ut_{\lambda}w^{-1}$ is an element in the extended affine Weyl group $\widehat{W}$, 
and $c^T(\mathcal{N})$ is the $T$-equivariant total Chern class of $\mathcal{N}$.
\end{Thm}

Taking the lowest degree terms on both sides leads to a connection concerning the cohomology classes of   projected Richardson varieties and   affine Schubert varieties studied by He and Lam \cite[Theorem 5.8]{HL15}, which, if further restricting to the type $A$ Grassmannian case, recovers  \cite[Theorem 7.8]{KLS-Posi}. 
Indeed, the above implication  is one of our main motivations of this work. It is also worth mentioning that our proof differs   from that  used in  \cite{HL15} and seems  simpler. The proof in \textit{loc. cit.} depends heavily on the Billey-type localization formulae for  Schubert classes. 
Instead, we employ the left and right Demazure--Lusztig operators to deduce that both sides satisfy the same recurrences, which uniquely determine these classes. We believe that this strategy should still work in the equivariant $K$-theory setting.

We next switch to the combinatorial sides of this paper. The study of characteristic functions of open Richardson varieties and their CSM classes comes with the  \emph{extended $P$-Bruhat order} $\leq_P$ on $W$. This is exhibited  in Theorem \ref{thm:geomcrit}:
$$\mathsf{Fun}(G/P)\ni \pi_*(\mathbbm{1}_{\mathring{R}_{u,w}})\neq 0\iff u\leq_P w.$$
As a comparison, it was shown in 
\cite[Lemma 3.1 and Corollary 3.4]{KLS-Proj} that 
$$H^*(G/P)\ni \pi_*([R_{u,w}])\neq 0\iff u\leq_P'w,$$
where  $\leq_P'$ is the ordinary $P$-Bruhat order.  
A new feature of the extended $P$-Bruhat order is the  $W_P$-invariance property. 
In fact, Theorem \ref{thm:charofPBru} implies that it is the strongest $W_P$-invariant partial order on $W$ weaker than the Bruhat order. 

When restricting to the Grassmannian  $G/P=\Gr_k(\mathbb{C}^n)$,  as aforementioned, (open) projective Richardson varieties are  known as  (open) positroid varieties, which are indexed by bounded affine permutations $f\in \tilde{S}_n$ by \cite{KLS-Posi}. As our another main result, we explicitly obtain a symmetric rational function representative $\tilde{F}_f(x_1,\ldots,x_k;y_1,\ldots,y_n)$, which is symmetric in $x_1,\ldots, x_k$, for the SM class
 of the open positroid variety $\mathring{\Pi}_{f}$.

\begin{Thm}[Theorem \ref{thm:Ff=SSM}]\label{BB-BN}
For  a bounded affine permutation $f$, we have
$$\SSM(\mathring{\Pi}_{f})=
\tilde{F}_f(x_1,\ldots,x_k;y_1,\ldots,y_n)\in 
H_T^*(\Gr_k(\mathbb{C}^n))_{\loc},$$
where $x_1,\ldots,x_k$ are the Chern roots of the dual of the tautological bundle over $\Gr_k(\mathbb{C}^n)$, and $y_1,\ldots,y_n\in H_T^*(\pt)$ are the standard equivariant parameters. 
\end{Thm}

The function $\tilde{F}_f$ is constructed via a weighted counting of certain colored string diagrams, which obviously admit a pipe dream realization. 
The proof of Theorem \ref{BB-BN} relies on a localization formula for SM classes along with  a diagrammatic calculus.

If we focus on the lowest degree terms on both sides of the formula in Theorem \ref{BB-BN}, then by comparing   with a recent  work (in progress) of Shimozono and Zhang \cite{SZ2025}, the right-hand side will exactly be 
the double affine Stanley symmetric function in \cite{LLS21}.
Therefore, if further letting $y_1=\cdots=y_n=0$, the right-hand side becomes the ordinary 
affine Stanley symmetric function of Lam \cite{Lam06}, and thus in this case Theorem \ref{BB-BN} specializes to  \cite[Theorem 7.1]{KLS-Posi}.

This paper is organized  as follows. Section \ref{OKUY} is devoted to the geometric background. In Section  \ref{JUY-i}, we define the extended $P$-Bruhat order and give several equivalent characterizations. 
Section \ref{finitepart} and Section \ref{affinepart} provide recursions satisfied by the finite and affine Chern classes.
Section \ref{comparison} finishes the proof of Theorem \ref{thm-B}. Finally, in Section \ref{combinpart}, our task is to prove Theorem 
\ref{BB-BN}. This is achieved by combining a localization formula for SM classes with diagrammatic computations on string diagrams.

\subsection*{Acknowledgments}
We would like to thank Leonardo C. Mihalcea, Hiroshi Naruse, Mark Shimozono, Sylvester W. Zhang, and Changlong Zhong for useful discussions. We also thank the referee for the valuable suggestions. CS is supported by the National Key R\&D Program of China (No. 2024YFA1014700). NF was supported by the National Natural Science Foundation of China (No. 12471314). PG was supported by the National Natural Science Foundation of
China (No. 12371329) and the Fundamental Research Funds for the Central
Universities (No. 63243072). RX acknowledges the partial support from the NSERC Discovery grant RGPIN-2022-03060, Canada.

\section{Preliminaries}\label{OKUY}

\subsection{Segre classes and Chern classes}\label{sect--2}

Let $X$ be a complex algebraic variety. The group of constructible functions $\mathsf{Fun}(X)$ consists of functions $\varphi = \sum_W c_W \mathbbm{1}_W$, where the sum is over a finite set of constructible subsets $W \subset X$, $c_W \in \mathbb{Z}$ are integers, and $\mathbbm{1}_W$ is the characteristic function of $W$. For a proper morphism $f\colon Y\rightarrow X$, there is a linear map $f_*\colon \mathsf{Fun}(Y)\rightarrow \mathsf{Fun}(X)$, such that for any constructible subset $W\subset Y$, 
\begin{equation}\label{equ:pushcons}
    f_*(\mathbbm{1}_W)(x)=\chi_{\rm top}(f^{-1}(x)\cap W),\
\end{equation}
where $x\in X$ and $\chi_{\text{top}}$ denotes the topological Euler characteristic.
Thus, $\mathsf{Fun}$ can be considered as a (covariant) functor from the category of complex algebraic varieties and proper morphisms to the category of abelian groups.

According to a conjecture attributed to Deligne and Grothendieck (see \cite[Note 87\(_1\)]{Grothen}), there is a unique natural transformation $c_*\colon  \mathsf{Fun} \to H_*$ from the functor of constructible functions on a complex algebraic variety $X$ to the Borel--Moore homology functor, where all morphisms are proper, such that if $X$ is smooth then $c_*(\mathbbm{1}_X)=c(TX)\cap [X]$, {where $c(TX)$ denotes the total Chern class of the tangent bundle $TX$ and $[X]$ denotes the fundamental class}.  This conjecture was proved by MacPherson \cite{MacPherson}; the class 
$c_*(\mathbbm{1}_X)$ for possibly singular $X$ was shown to coincide with a class defined earlier by Schwartz \cite{Sch65a, Sch65b, BS}. 

The theory of CSM classes was later extended to the equivariant setting by Ohmoto \cite{Ohmoto}. If $X$ has an 
action of a torus $T$, Ohmoto defined the group $\mathsf{Fun}^T(X)$ of {\em equivariant\/} constructible functions. Ohmoto \cite[Theorem 1.1]{Ohmoto} proves  that there is an equivariant version of MacPherson transformation $c_*^T\colon  \mathsf{Fun}^T(X) \to H_*^T(X)$ that satisfies $c_*^T(\mathbbm{1}_X) = c^T(TX) \cap [X]_T$ if $X$ is a non-singular variety, and that is functorial with respect to proper push-forwards. The last statement means that for all proper $T$-equivariant morphisms $Y\to X$ the following diagram
commutes:
$$\xymatrix{ 
\mathsf{Fun}^T(Y) \ar[r]^{c_*^T} \ar[d]_{f_*^T} & H_*^T(Y) \ar[d]^{f_*^T} \\
\mathsf{Fun}^T(X) \ar[r]^{c_*^T} & H_*^T(X).}$$ 
If $X$ is smooth, we will identify the (equivariant) homology and cohomology groups by Poincar{\'e} duality: $H_*^T(X)\simeq H^*_T(X)$.

\begin{Defin}
Let $Z$ be a $T$-invariant constructible subvariety of $X$. 
\begin{enumerate}
\item
We denote by $\CSM(Z):=c_*^T(\mathbbm{1}_{Z}) {~\in H_*^T(X)}$ the {\em equivariant Chern--Schwartz--MacPherson (CSM) class\/} of $Z$. 
\item 
If $X$ is smooth, we denote by $\SSM(Z):=\frac{c_*^T(\mathbbm{1}_{Z})}{c^T(T X)} ~\in \widehat{H}^*_T(X)$ 
the {\em equivariant Segre--MacPherson (SM) class\/} of $Z$, where $\widehat{H}^*_T(X)$ is an appropriate completion
of $H^*_T(X)$. 
\end{enumerate}
\end{Defin}

\subsection{Affine Flag Varieties}
Let $X_*(T)$ be the cocharacter lattice of $T$, and let 
$$\widehat{W}:=W\ltimes X_*(T)$$ be the \emph{extended affine Weyl group}.
For a cocharacter $\lambda\in X_*(T)$, we denote by $t_\lambda$ the corresponding element in $\widehat{W}$. Note that in $\widehat{W}$, we have 
$$wt_\lambda w^{-1}=t_{w\lambda},\qquad w\in W,\lambda\in X_*(T).$$
Denote the coroot lattice as $Q^\vee\subseteq X_*(T)$. 
It is known that the subgroup $W_a:=W\ltimes Q^\vee$, called the \emph{affine Weyl group}, is the Coxeter group of the corresponding affine Dynkin diagram \cite{Kac} with generators
$$s_i\in W\, (i\in I),\quad\text{ and \quad} s_0=t_{\theta^\vee}s_\theta,
$$
where $\theta$ is the highest root. 
Let $R$ be the associated root system of $W$, and   $R^+$ (resp., $R^-$) be the set of positive roots $\alpha>0$ (resp., negative roots $\alpha<0$).
The length function on $W_a$ can be extended to $\widehat{W}$, which owns  the Iwahori--Matsumoto  formula \cite{IwaMat}:
$$\ell(wt_\lambda)
=
\sum_{\alpha>0,w\alpha>0}\big|\langle \alpha,\lambda\rangle\big|+
\sum_{\alpha>0,w\alpha<0}\big|\langle \alpha,\lambda\rangle+1\big|.
$$ 

Let $\mathbb{C}[[z]]$ (resp. $\mathbb{C}(\!(z)\!):=\mathbb{C}[[z]][z^{-1}]$) be the formal power series ring (resp. formal Laurent series ring), and let $G[[z]]$ (resp. $G(\!(z)\!)$)  be the $\mathbb{C}[[z]]$-points (resp. $\mathbb{C}(\!(z)\!)$-points) of $G$. There is an evaluation at $z=0$ map from $G[[z]]$ to $G$, and let $\mathcal{I}$ be the inverse image of the Borel subgroup $B$. Then the \emph{affine flag variety} is
$$\Fl_G = G(\!(z)\!)/\mathcal{I},$$
whose $T$-fixed points
$(\Fl_G)^T$ can be identified with  $\widehat{W}$ as follows. Any cocharacter $\lambda\in X_*(T)$ defines a morphism $\mathbb{C}(\!(z)\!)^*\rightarrow T(\!(z)\!)\subset G(\!(z)\!)$, we use $z^\lambda\in G(\!(z)\!)$ to denote the image of $z$. For any $w\in W$, let $\dot{w}$ denote a lift of it in $G[[z]]$. Then for each $wt_\lambda\in \widehat{W}$, the corresponding fixed point in $\Fl_G$ is $\dot{w}z^{-\lambda}\mathcal{I}\in \Fl_G$, which will be just denoted by $wt_\lambda$ throughout the paper. Let $\mathring{\Sigma}_{wt_\lambda}:=\mathcal{I}\dot{w}z^{-\lambda}\mathcal{I}/\mathcal{I}$ be the {\it affine Schubert cell} of dimension $\ell(wt_\lambda)$. The  affine flag variety has a cell decomposition 
\[\Fl_G=\bigsqcup_{wt_\lambda\in \widehat{W}}\mathring{\Sigma}_{wt_\lambda}.\]
For each $n\geq 0$, let 
\[
X_n:=\bigsqcup_{wt_\lambda\in \widehat{W},\,  \ell(wt_\lambda)\leq n}\mathring{\Sigma}_{wt_\lambda}.    
\]
 Then $\Fl_G$ is the increasing limit of $X_n$, which is called an ind-variety.

There is a loop rotation action of $\mathbb{C}^*_{\rot}:=\mathbb{C}^*$ on $\Fl_G$ by scaling the parameter $z$. Let $\delta$ be the degree one character of this action, which is also the imaginary root for the corresponding Kac--Moody Lie algebra $\hat{\mathfrak{g}}$. Then the positive (real)  affine roots are 
\begin{equation}\label{equ:affpos}
    R^+_{\aff}=\{\alpha+k\delta\mid \alpha\in R^+, k\geq 0 \text{ or } \alpha\in R^-, k\geq 1\},
\end{equation}
and the negative (real) affine roots are 
\[R^-_{\aff}=\{\alpha+k\delta\mid \alpha\in R^+, k\leq -1 \text{ or } \alpha\in R^-, k\leq 0\}.\]
The extended affine Weyl group $\widehat{W}$ acts on the lattice $X^*(T)\oplus\mathbbm{Z}\delta$ by the following formula (see \cite[(3.1)]{LS10})
\begin{equation}\label{equ:affact}
    wt_\lambda(\mu+k\delta)=w(\mu)+(k-\langle{\lambda, \mu\rangle})\delta,
\end{equation}
where $\mu\in X^*(T)$.

\begin{Lemma}\label{lem:tangentwt}
    For any dominant cocharacter $\lambda$, the torus $T\times \mathbb{C}^*_{\rot}$ weights of the tangent space $T_{t_\lambda}(\mathring{\Sigma}_{t_\lambda})$ are 
    \[\{\alpha+k\delta\mid \alpha\in R^-, 1\leq k\leq -\langle\lambda,\alpha\rangle\}.\]
\end{Lemma}
\begin{proof}
    Since $\mathring{\Sigma}_{t_\lambda}\simeq \mathcal{I}/(\mathcal{I}\cap z^{-\lambda} \mathcal{I} z^{\lambda})$, the desired weights are 
    \[R_{\aff}^+\setminus (R_{\aff}^+\cap \Ad_{t_\lambda}(R_{\aff}^+)).\]
    By  \eqref{equ:affpos} and \eqref{equ:affact},
    \[R_{\aff}^+\cap \Ad_{t_\lambda}(R_{\aff}^+)=\{\alpha+k\delta\mid \alpha\in R^+, k\geq 0,\text{ or } \alpha\in R^-, k\geq 1-\langle\lambda,\alpha\rangle \}.\]
    Hence, the lemma holds.
\end{proof}

Let $\Gr_G:= G(\!(z)\!)/G[[z]]$ be the {\it affine Grassmannian}. There is a natural projection map $\Fl_G\rightarrow
\Gr_G$ whose fibers are isomorphic to $G/B$. The $T$-fixed points $(\Gr_G)^T$ are in bijection with the coset $\widehat{W}/W\simeq X_*(T)$. For any $\lambda\in X_*(T)$, we let $t_\lambda W$ represent the fixed point $z^{-\lambda}G[[z]]/G[[z]]$. For any dominant cocharacter $\lambda$, let $\Gr_\lambda:=G[[z]]z^{-\lambda} G[[z]]/G[[z]]\subset \Gr_G$. Then 
\[\Gr_G=\bigsqcup_{\lambda\in X_*(T)^+}\Gr_\lambda.\]
Let $\ev:G[[z]]\rightarrow G$ be the evaluation at $z=0$ map. Recall that the parabolic subgroup $P$ containing the positive Borel subgroup $B$ is associated with the simple roots 
$\{\alpha_i\mid \langle \lambda,\alpha_i^\vee\rangle=0\}$. By  \eqref{equ:affact},
\[\ev\bigg(G[[z]]\cap z^{-\lambda} G[[z]]z^\lambda\bigg)=P\subset G.\]
Hence,
\[\Gr_\lambda\simeq G[[z]]/(G[[z]]\cap z^{-\lambda} G[[z]]z^\lambda)\]
maps to $G/P$ via the map $\ev$. Moreover, it
is a bundle of affine spaces over $G/P$, 
and there is a closed embedding 
\[i_\lambda:G/P\simeq G\cdot z^{-\lambda} G[[z]]/G[[z]]\subset \Gr_\lambda,\]
by regarding $G$ as a subgroup of $G[[z]]$ of constant loops, see \cite{Zhu}.
\begin{Lemma}\label{lem:tanweGr}
    The $T\times\mathbb{C}^*_{\rot}$ weights of the tangent space $T_{t_{\lambda}W}\Gr_\lambda$ is
    \[\{-\alpha+k\delta\mid 0\leq k< \langle\lambda,\alpha\rangle\}.\]
\end{Lemma}
\begin{proof}
    By $\Gr_\lambda=G[[z]]/(G[[z]]\cap z^{-\lambda} G[[z]]z^\lambda)$, and  \eqref{equ:affact},
    the desired weights are 
    \begin{align*}
        &\quad \{\alpha+k\delta\mid \alpha \in R, k\geq 0\}\setminus \{\alpha+k\delta\mid \alpha \in R^+, k\geq 0, \text{ or }\alpha\in R^-, k\geq -\langle\lambda,\alpha\rangle\}\\
        &=\{\alpha+k\delta\mid \alpha \in R^-, 0\leq k< -\langle\lambda,\alpha\rangle\}\\
        &=\{-\alpha+k\delta\mid \alpha \in R^+, 0\leq k< \langle\lambda,\alpha\rangle\}.
    \qedhere
    \end{align*}
\end{proof}

\section{Extended $P$-Bruhat order}\label{JUY-i}

In this section, we define the {extended $P$-Bruhat order} on a Weyl group and give several  characterizations for this order. 

\subsection{Extended $P$-Bruhat order}
 For a parabolic subgroup  $W_P$ of $W$, denote by $R^+_P$ the associated  positive system of $W_P$. 
For $u, w\in W$, write $u\stackrel{P}\to w$ to mean that $w=us_{\alpha}$ for some $\alpha\in R^+\setminus R^+_P$ such that $\ell(w)>\ell(v)$. 
Notice that the condition $\ell(w)>\ell(v)$ is equivalent to saying that $u(\alpha)\in R^+$. 
The \emph{extended $P$-Bruhat order} $\leq_P$ on $W$ is  the transitive closure of the relations $u\stackrel{P}\to w$, 
that is,
$$u\leq_P w\iff \text{there exists a path }
u=u_0\stackrel{P}\to u_1\stackrel{P}\to \cdots\stackrel{P}\to u_{k-1}\stackrel{P}\to u_k= w.$$

\begin{Rmk}
In the definition of $u\stackrel{P}\to w$, if replacing the condition $\ell(w)>\ell(u)$ by $\ell(w)=\ell(u)+1$, then the transitive closure forms the ordinary $P$-Bruhat order as introduced in \cite{KLS-Proj}. To distinguish, we shall denote the ordinary $P$-Bruhat order by $\leq'_P$. 
It is clear that $u\leq'_P w$ implies $u\leq_P w$. 
\end{Rmk}

\begin{Eg}
Let us consider the rank $2$ case.
Then
$$W=\langle s,t|s^2=t^2=(st)^m=1\rangle,\qquad W_P=\langle s\rangle \subset W,$$
where $m=3, 4$, or $6$, i.e., $W$ is of type $A_2,B_2=C_2,G_2$, 
see Figure \ref{fig:rank2}.
\end{Eg}

\begin{figure}[hhhhhhhhhh]
$$
\begin{array}{ccc}
\begin{matrix}
    \begin{tikzpicture}
    \node (0) at (0, 0) {\(\id\)};
    \node (1) at (1.30, 0.75) {\(s\)};
    \node (2) at (-1.30, 0.75) {\(t\)};
    \node (12) at (1.30, 2.25) {\(st\)};
    \node (21) at (-1.30, 2.25) {\(ts\)};
    \node (121) at (0, 3) {\(w_0\)};
    \draw[->,thick] (0) edge (2); 
    \draw[->,gray] (0) edge (121); 
    \draw[->,thick] (1) edge (12); 
    \draw[->,thick] (1) edge (21); 
    \draw[->,thick] (2) edge (12); 
    \draw[->,thick] (21) edge (121); 
    \end{tikzpicture}
\end{matrix}&\begin{matrix}
\begin{tikzpicture}
\node (0) at (0, 0) {\(\id\)};
\node (1) at (1.41, 0.59) {\(s\)};
\node (2) at (-1.41, 0.59) {\(t\)};
\node (12) at (2, 2) {\(st\)};
\node (21) at (-2, 2) {\(ts\)};
\node (121) at (1.41, 3.41) {\(sts\)};
\node (212) at (-1.41, 3.41) {\(tst\)};
\node (2121) at (0, 4) {\(w_0\)};
\draw[->,thick] (0) edge (2); 
\draw[->,gray] (0) edge (212); 
\draw[->,gray] (0) edge (121); 
\draw[->,thick] (1) edge (12); 
\draw[->,gray] (1) edge (2121); 
\draw[->,thick] (1) edge (21); 
\draw[->,thick] (2) edge (12); 
\draw[->,gray] (2) edge (2121); 
\draw[->,thick] (12) edge (212); 
\draw[->,thick] (21) edge (212); 
\draw[->,thick] (21) edge (121); 
\draw[->,thick] (121) edge (2121); 
\end{tikzpicture}
\end{matrix}&\begin{matrix}
\begin{tikzpicture}
\node (0) at (0, 0) {\(\id\)};
\node (1) at (1.5, 0.40) {\(s\)};
\node (2) at (-1.5, 0.40) {\(t\)};
\node (12) at (2.60, 1.5) {\(st\)};
\node (21) at (-2.60, 1.5) {\(ts\)};
\node (121) at (3, 3) {\(sts\)};
\node (212) at (-3, 3) {\(tst\)};
\node (1212) at (2.60, 4.5) {\(stst\)};
\node (2121) at (-2.60, 4.5) {\(tsts\)};
\node (12121) at (1.5, 5.60) {\(ststs\)};
\node (21212) at (-1.5, 5.60) {\(tstst\)};
\node (212121) at (0, 6) {\(w_0\)};
\draw[->,thick] (0) edge (2); 
\draw[->,gray] (0) edge (212); 
\draw[->,gray] (0) edge (121); 
\draw[->,gray] (0) edge (12121); 
\draw[->,gray] (0) edge (21212); 
\draw[->,thick] (1) edge (12); 
\draw[->,gray] (1) edge (1212); 
\draw[->,thick] (1) edge (21); 
\draw[->,gray] (1) edge (2121); 
\draw[->,gray] (1) edge (212121); 
\draw[->,thick] (2) edge (12); 
\draw[->,gray] (2) edge (2121); 
\draw[->,gray] (2) edge (212121); 
\draw[->,gray] (2) edge (1212); 
\draw[->,gray] (12) edge (12121); 
\draw[->,gray] (12) edge (21212); 
\draw[->,thick] (12) edge (212); 
\draw[->,thick] (21) edge (212); 
\draw[->,gray] (21) edge (21212); 
\draw[->,thick] (21) edge (121); 
\draw[->,gray] (21) edge (12121); 
\draw[->,thick] (121) edge (1212); 
\draw[->,gray] (121) edge (212121); 
\draw[->,thick] (121) edge (2121); 
\draw[->,gray] (212) edge (212121); 
\draw[->,thick] (212) edge (1212); 
\draw[->,thick] (1212) edge (21212); 
\draw[->,thick] (2121) edge (21212); 
\draw[->,thick] (2121) edge (12121); 
\draw[->,thick] (12121) edge (212121); 
\end{tikzpicture}
\end{matrix}\\
m=3 & m = 4 & m=6
\end{array}
$$
\caption{Rank 2 extended $P$-Bruhat order}
    \label{fig:rank2}
\end{figure}
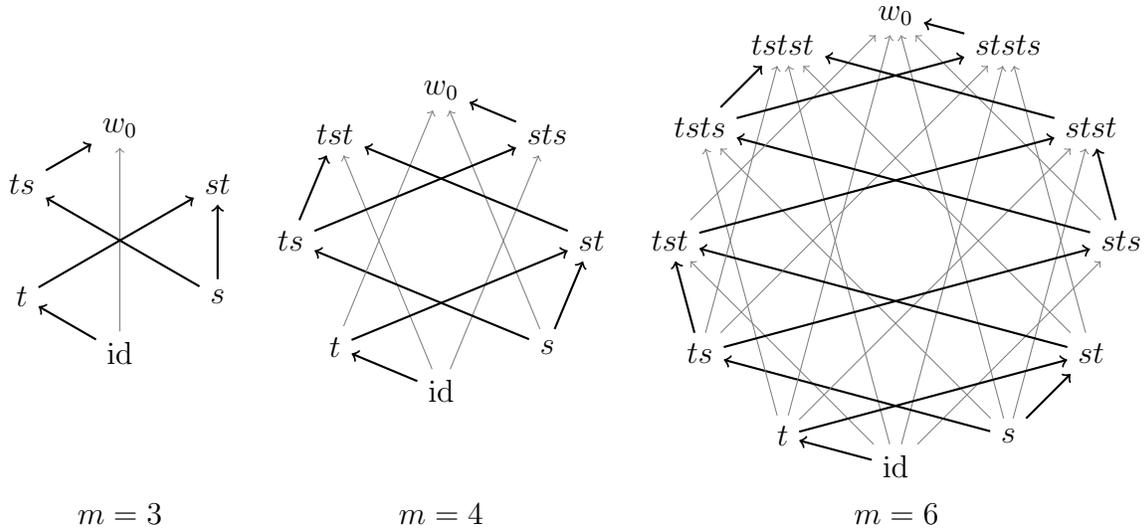

\begin{Eg}
The case $W=S_4$ and $W_P=S_1\times S_2\times S_1$ is illustrated in Figure \ref{fig:S4adj}.
\end{Eg}

\begin{figure}[hhhhhhhhhhh]
$$\begin{tikzpicture}
\node (1234) at (-0.750, 0.000) {\(\sf\scriptstyle1234\)};
\node (2134) at (2.25, 1.50) {\(\sf\scriptstyle2134\)};
\node (1324) at (0.750, 0.000) {\(\sf\scriptstyle1324\)};
\node (1243) at (-3.75, 1.50) {\(\sf\scriptstyle1243\)};
\node (2314) at (-2.25, 1.50) {\(\sf\scriptstyle2314\)};
\node (3124) at (2.25, 3.00) {\(\sf\scriptstyle3124\)};
\node (1342) at (-3.75, 3.00) {\(\sf\scriptstyle1342\)};
\node (2143) at (-0.750, 3.00) {\(\sf\scriptstyle2143\)};
\node (1423) at (3.75, 1.50) {\(\sf\scriptstyle1423\)};
\node (3214) at (-2.25, 3.00) {\(\sf\scriptstyle3214\)};
\node (2341) at (-3.75, 4.50) {\(\sf\scriptstyle2341\)};
\node (3142) at (-0.750, 4.50) {\(\sf\scriptstyle3142\)};
\node (1432) at (3.75, 3.00) {\(\sf\scriptstyle1432\)};
\node (2413) at (0.750, 3.00) {\(\sf\scriptstyle2413\)};
\node (4123) at (2.25, 4.50) {\(\sf\scriptstyle4123\)};
\node (3241) at (-3.75, 6.00) {\(\sf\scriptstyle3241\)};
\node (2431) at (3.75, 4.50) {\(\sf\scriptstyle2431\)};
\node (3412) at (0.750, 4.50) {\(\sf\scriptstyle3412\)};
\node (4132) at (2.25, 6.00) {\(\sf\scriptstyle4132\)};
\node (4213) at (-2.25, 4.50) {\(\sf\scriptstyle4213\)};
\node (3421) at (3.75, 6.00) {\(\sf\scriptstyle3421\)};
\node (4231) at (-0.750, 7.50) {\(\sf\scriptstyle4231\)};
\node (4312) at (-2.25, 6.00) {\(\sf\scriptstyle4312\)};
\node (4321) at (0.750, 7.50) {\(\sf\scriptstyle4321\)};
\draw[->,thick] (1234) edge (2134); %
\draw[->,thick] (1234) edge (1243); %
\draw[->,gray] (1234) edge (3214); %
\draw[->,gray] (1234) edge (1432); %
\draw[->,gray] (1234)  edge[bend left= 15] (4231); %
\draw[->,thick] (2134) edge (2143); %
\draw[->,thick] (2134) edge (3124); %
\draw[->,gray] (2134) edge (2431); %
\draw[->,gray] (2134) edge [bend left= 30] (4132); %
\draw[->,thick] (1324) edge (3124); %
\draw[->,thick] (1324) edge (1342); %
\draw[->,thick] (1324) edge (2314); %
\draw[->,thick] (1324) edge (1423); %
\draw[->,gray] (1324) edge [bend right= 15] (4321); %
\draw[->,thick] (1243) edge (2143); %
\draw[->,gray] (1243) edge (4213); %
\draw[->,thick] (1243) edge (1342); %
\draw[->,gray] (1243) edge [bend left= 30] (3241); %
\draw[->,thick] (2314) edge (3214); %
\draw[->,thick] (2314) edge (2341); %
\draw[->,thick] (2314) edge (2413); 
\draw[->,gray] (2314) edge[bend right = 30] (4312); %
\draw[->,thick] (3124) edge (3142); %
\draw[->,gray] (3124) edge (3421); %
\draw[->,thick] (3124) edge (4123); %
\draw[->,thick] (1342) edge (3142); %
\draw[->,gray] (1342) edge (4312); %
\draw[->,thick] (1342) edge (2341); %
\draw[->,thick] (2143) edge (4123); %
\draw[->,thick] (2143) edge (2341); %
\draw[->,thick] (2143) edge (3142); %
\draw[->,thick] (1423) edge (4123); %
\draw[->,thick] (1423) edge (1432); %
\draw[->,thick] (1423) edge (2413); %
\draw[->,gray] (1423) edge[bend right = 30] (3421); %
\draw[->,thick] (3214) edge (3241); %
\draw[->,thick] (3214) edge (3412); %
\draw[->,thick] (3214) edge (4213); %
\draw[->,thick] (2341) edge (3241); %
\draw[->,gray] (2341) edge (4321); %
\draw[->,thick] (3142) edge (4132); %
\draw[->,thick] (3142) edge (3241); %
\draw[->,thick] (1432) edge (4132); %
\draw[->,thick] (1432) edge (3412); %
\draw[->,thick] (1432) edge (2431); %
\draw[->,thick] (2413) edge (4213); %
\draw[->,thick] (2413) edge (2431); %
\draw[->,thick] (2413) edge (3412); %
\draw[->,thick] (4123) edge (4132); %
\draw[->,gray] (4123) edge (4321); %
\draw[->,thick] (3241) edge (4231); %
\draw[->,thick] (2431) edge (4231); %
\draw[->,thick] (2431) edge (3421); %
\draw[->,thick] (3412) edge (4312); %
\draw[->,thick] (3412) edge (3421); %
\draw[->,thick] (4132) edge (4231); %
\draw[->,thick] (4213) edge (4231); %
\draw[->,thick] (4213) edge (4312); %
\draw[->,thick] (3421) edge (4321); %
\draw[->,thick] (4312) edge (4321); %
\end{tikzpicture}$$
    \caption{$W=S_4$ and $W_P = S_1\times S_2\times S_1$}
    \label{fig:S4adj}
\end{figure}
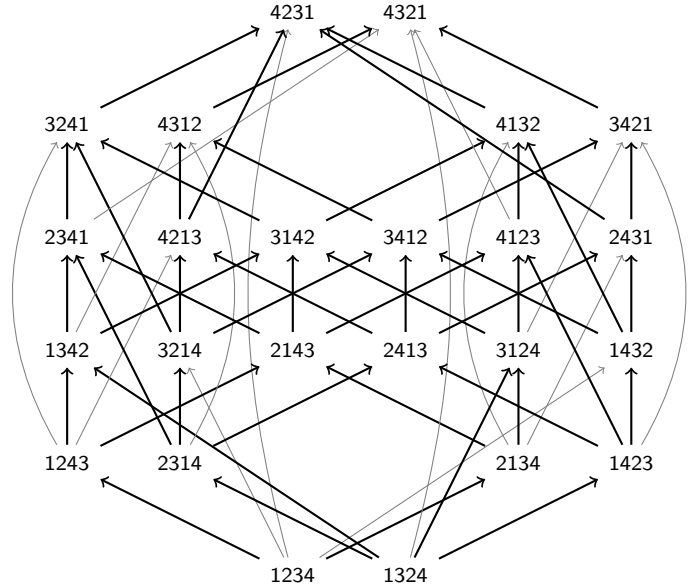

\begin{Eg}\label{eg:extkBru}
When $W=S_n$ is the Weyl group of type $A_{n-1}$  and $W_P=S_k\times S_{n-k}$ is a maximal parabolic subgroup, we  have 
$$\{s_\alpha:\alpha\in R^+\setminus R_P^+\}=\{\tau_{ab}:1\leq a\leq k<b\leq n\},$$
where $\tau_{ab}\in S_n$ is the tranposition of $a$ and $b$. 
In this case,  the extended $P$-Bruhat order is the \emph{extended $k$-Bruhat order $\leq_k$} investigated  in \cite[Lemma 6.4]{FGX}, which has the following combinatorial description: 
$$u\leq_k w\iff 
\begin{cases}
\forall\, 1\leq a\leq k, \, u(a)\leq w(a),\\
\forall\, k<b\leq n, \, u(b)\geq w(b)
.\end{cases}$$
\end{Eg}

\begin{Prop}\label{prop:charofPBru}
Let $u,w\in W$. 
\begin{enumerate}[\rm\quad (1)]
    \item \label{item:charofBru1}
    For   $v\in W_P$, we have 
    $u\leq_P w\iff uv\leq_P wv$. 
    \item \label{item:charofBru2}
  For  $w\in W^P$, we have 
    $u\leq_P w\iff u\leq w$. 
\end{enumerate}
\end{Prop}
\begin{proof}
\eqref{item:charofBru1}
Clearly, we only need to check that $u\leq_P w\,\Longrightarrow\, uv\leq_P wv$ for any $v\in W_P$. It suffices to verify $u\stackrel{P}\to w\,\Longrightarrow\, uv\stackrel{P}\to wv$.
Assume that 
$w=us_\alpha$ for some $\alpha\in R^+\setminus R_P^+$ with $u(\alpha)\in R^+$. 
Notice that
\[
wv=us_\alpha v=uvv^{-1}s_\alpha v=uvs_{v^{-1}(\alpha)}.
\] 
Since $v\in W_P$ and $\alpha\in R^+\setminus R_P^+$, we have  $v^{-1}(\alpha)\in R^+\setminus R_P^+$. 
This, together with the fact that $uv(v^{-1}(\alpha))=u(\alpha)\in R^+$, implies   $uv\stackrel{P}\to wv$. 

\eqref{item:charofBru2}
For any $w\in W$, it is obvious that 
$u\leq_P w\Longrightarrow u\leq w.$ We check the reverse direction. Suppose that $u\leq w$ with $w\in W^P$. Then  $u\leq'_P w$ in the ordinary    $P$-Bruhat order  \cite[Proposition 2.5]{KLS-Proj}. This yields that  $u\leq_P w$. 
\end{proof}

Combining the above gives the following  equivalent statements.  

\begin{Th}\label{thm:charofPBru}
Let $u,w\in W$. 
Then the following are equivalent:
\begin{enumerate}[\rm \quad (1)]
    \item $u\leq_P w$; 
    \item $uv\leq_P wv$ for any $v\in W_P$; 
    \item $uv\leq wv$ for some $v\in W_P$ such that $wv\in W^P$.
\end{enumerate}
\end{Th}

\begin{Eg}
Recall that the Coxeter group of type $BC_n$ can be realized as 
$$W=\{w\in S_{\{\pm 1,\ldots,\pm n\}}:w(-i)=-w(i)\}.$$
We consider the parabolic subgroup $W_P\cong S_n$ of $w\in W$ such that  $w(i)>0$ for $i>0$. Then we have
\begin{equation}\label{eq:typeBnBruord}
u \leq_P w\iff \forall 1\leq i\leq n,\, u(i)\geq w(i).
\end{equation}
To see this, we consider another Weyl group $\mathcal{W}=S_{2n}$ with parabolic subgroup $\mathcal{W}_{P}=S_n\times S_n$.
We have a natural embedding
$i:W\subset \mathcal{W}$ such that 
$$u\leq w\iff i(u)\leq i(w), $$
see \cite[Section 8.1]{BB}. 
Moreover, it is direct to check that
$$
w\in W_P\iff i(w)\in \mathcal{W}_{P},\qquad 
w\in W^P\iff i(w)\in \mathcal{W}^{P}.$$
By Theorem \ref{thm:charofPBru}, we have
$$u\leq_P w\iff i(u)\leq_{P} i(w).$$
By Example \ref{eg:extkBru}, we get \eqref{eq:typeBnBruord}. 
\end{Eg}

\subsection{Affine characterization}
Let us fix a dominant cocharacter $\lambda\in X_*(T)$. 
For $u,w\in W$, we define 
\begin{equation}\label{eq:fuw}
    f_{u,w}^\lambda = ut_{\lambda}w^{-1}\in \widehat{W}.
\end{equation}
Assume that the stabilizer of $\lambda$ in $W$ is the parabolic subgroup $W_P$. Let $W\lambda$ be the Weyl group orbit of $\lambda$.

 \begin{Th}\label{thm:affinecrit}
For $u,w\in W$, we have
$$u\leq_P w\iff f_{u,w}^{\lambda}\leq t_{\mu} \text{ for some $\mu\in W\lambda$}.$$
\end{Th}

To give a proof of Theorem \ref{thm:affinecrit}, we need the following  length formula.

\begin{Lemma}\label{lem:ellutw}
If $w\in W^P$, then 
$$\ell(f^{\lambda}_{u,w}) = \ell(u)+\ell(t_\lambda)-\ell(w).$$
\end{Lemma}
\begin{proof}
It is known \cite[Lemma 3.3]{LS10} that  for $w\in W^P$,   $t_\lambda w^{-1}$ is the minimal representative of the right coset $Wt_\lambda w^{-1}$.
So 
\begin{equation}\label{Gydhso}
\ell(f^\lambda_{u,w})=\ell(ut_{\lambda} w^{-1}) = \ell(u)+\ell(t_\lambda w^{-1}).   
\end{equation}
Since 
\[
  \ell(f^\lambda_{w,w})
=\ell(t_{w\lambda})=\ell(t_{\lambda}) = \ell(f^{\lambda}_{\id,\id})= \sum_{\alpha\in R^+}|\langle\lambda,\alpha\rangle|,  
\]  we obtain that
$\ell(w)+\ell(t_{\lambda} w^{-1})=\ell(t_{\lambda})$, 
which, together with \eqref{Gydhso}, leads to the desired formula in the lemma. 
\end{proof}

\begin{proof}[Proof of Theorem \ref{thm:affinecrit}]
Notice  that when $v\in W_P$, 
$$
f^\lambda_{uv,wv}=uvt_\lambda v^{-1}w = ut_\lambda w^{-1}=f^\lambda_{u,w}.$$
Hence, according to (1) in Proposition \ref{prop:charofPBru}, we may assume $w\in W^P$. In this situation, by (2) in Proposition \ref{prop:charofPBru}, $u\leq_P w$ is equivalent to $u\leq w$. 
If $u\leq w$, 
by Lemma \ref{lem:ellutw}, the decomposition $f^{\lambda}_{u,w}=u \cdot t_\lambda w^{-1}$ is reduced. So 
$$f^\lambda_{u,w}=ut_\lambda w^{-1}\leq wt_\lambda w^{-1}=t_{w\lambda}.$$
Conversely, assume that $\mu=w_+\lambda$ for some $w_+\in W^P$. 
Then 
$$f_{u,w}^\lambda = ut_\lambda w^{-1}\leq t_{\mu}= w_+ t_\lambda w_+^{-1}=f^{\lambda}_{w^+,w^+}.$$
By \cite[Proposition 2.1]{HL15}\footnote{Note that $t^{-\lambda}$ in \textit{loc. cit.} is denoted by $t_\lambda$ here.}, there exists $v\in W_P$ such that 
$$u\leq w_+v,\qquad w\geq w_+v.$$
So we have $u\leq w$.
\end{proof}

\subsection{Geometric characterization}
Recall that for a constructible subset $Y$ of a complex algebraic variety $X$, 
$\mathbbm{1}_{Y}\in \mathsf{Fun}(X)$ denotes its characteristic function. 

\begin{Lemma}\label{lem:WPinvof1uw}
For any $v\in W_P$ and $u, w\in W$, we have
$$\pi_*(\mathbbm{1}_{\mathring{R}_{u,w}})=\pi_*(\mathbbm{1}_{\mathring{R}_{uv,wv}}).$$
\end{Lemma}
\begin{Rmk}
From the proof below, we see that the equality \(\pi(\mathring{R}_{u,w})= \pi(\mathring{R}_{uv,wv})\) does not hold in general.
\end{Rmk}

\begin{proof}[Proof of Lemma \ref{lem:WPinvof1uw}]
First of all, we can assume $w\in W^P$.
It suffices to show when $v=s_i$ is a simple reflection in $W_P$. 
Now the projection $\pi$ factorizes into 
$$G/B\stackrel{\pi_i}\longrightarrow G/P_i\stackrel{\rho_i}\longrightarrow G/P$$
for $P_i=B\cup Bs_iB$ the minimal parabolic subgroup corresponding to $i\in I$. The Lemma follows if $\pi_{i*}(\mathbbm{1}_{\mathring{R}_{u,w}})=\pi_{i*}(\mathbbm{1}_{\mathring{R}_{us_i,ws_i}})$. Therefore, the
Lemma is further  reduced to the case when $P=P_i$ and $\pi=\pi_i$. Since $w\in W^P$, the case of $u\in W^P$ is proved in \cite[Lemma 3.1]{KLS-Proj}.

Now we prove the other case $us_i>u$. Note that $\pi_i$ is a $\mathbb{P}^1$ bundle. 
At any point $z\in G/P_i$, we need to show 
$$\chi\big(\mathring{R}_{u,w}\cap \pi_i^{-1}(z)\big)=
\chi\big(\mathring{R}_{us_i,ws_i}\cap \pi_i^{-1}(z)\big).$$
The intersection of the Schubert cells $\mathring{\Sigma}_w$, $\mathring{\Sigma}_{ws_i}$ and the fibre $\pi_i^{-1}(z)\cong \mathbb{P}^1$ has two possibilities
\begin{itemize}
    \item[(1)] $\mathring{\Sigma}_{w} \cap \pi_i^{-1}(z)=\{p\}$ a point, 
    and $\mathring{\Sigma}_{ws_i}\cap \pi_i^{-1}(z)=\pi_i^{-1}(z)\setminus \{p\}$ an affine line; 
    \item[(2)] both intersections are empty.
\end{itemize}
Similarly, the intersection of the opposite Schubert cells $\mathring{\Sigma}^{u}$, $\mathring{\Sigma}^{us_i}$ and the fibre $\pi_i^{-1}(z)\cong \mathbb{P}^1$ has two possible possibilities
\begin{itemize}
    \item[(a)] $\mathring{\Sigma}^{u} \cap \pi_i^{-1}(z)=\{q\}$ a point, and $\mathring{\Sigma}^{us_i}\cap \pi_i^{-1}(z)=\pi_i^{-1}(z)\setminus \{q\}$ an affine line; 
    \item[(b)] both intersections are empty.
\end{itemize}
In the case (2) or (b), we have 
$$\mathring{R}_{u,w} \cap \pi_i^{-1}(z)=\mathring{R}_{us_i,ws_i} \cap \pi_i^{-1}(z)=\varnothing. $$
Hence, we only need to deal with case (1) and (a). When $p=q$, we have 
$$
\mathring{R}_{u,w} \cap \pi_i^{-1}(z)=\{p\},\qquad 
\mathring{R}_{us_i,ws_i} \cap \pi_i^{-1}(z)=\pi_i^{-1}(z)\setminus \{p\}\simeq \mathbb{C};$$
when $p\neq q$, we have 
$$
\mathring{R}_{u,w} \cap \pi_i^{-1}(z)=\varnothing,\qquad 
\mathring{R}_{us_i,ws_i} \cap \pi_i^{-1}(z)=\pi_i^{-1}(z)\setminus \{p,q\}\simeq \mathbb{C}^\times.$$
In both cases, they have the same Euler characteristics. 
This finishes the proof. 
\end{proof}

\begin{Th}\label{thm:geomcrit}
For $u, w\in W$, we have 
$$\pi_*(\mathbbm{1}_{\mathring{R}_{u,w}})\neq 0\iff u\leq_P w.$$
\end{Th}
\begin{proof}
By Lemma \ref{lem:WPinvof1uw} and Proposition \ref{prop:charofPBru}, we can assume $w\in W^P$. 
In this case, the restriction of $\pi$ to $\mathring{\Sigma}_w$ is injective, thus $\pi_*(\mathbbm{1}_{\mathring{R}_{u,w}})\neq 0$ if and only if $\mathring{R}_{u,w}\neq \varnothing$, i.e. $u\leq w$. 
Hence, the theorem follows from Theorem \ref{thm:charofPBru}. 
\end{proof}

\section{Recursion of Chern classes, finite part}\label{finitepart}

In this section, we will characterize the CSM classes of open projected Richardson varieties. 

\subsection{Chern classes of open Richardson varieties}\label{sec:csmRichard}
Let us first characterize the CSM class of (unprojected) open Richardson varieties.
Firstly, since the Schubert cells $\mathring{\Sigma}_w$ and $\mathring{\Sigma}^{u}$ intersect transversally, we have (see \cite{Sch17})
\begin{equation}
\label{eq:CSMR=CSMSSM}
\CSM(\mathring{R}_{u,w})=
\CSM(\mathring{\Sigma}_{w})\cdot \SSM(\mathring{\Sigma}^{u}).
\end{equation}

Let us recall some operators acting on $H_T^*(G/B)$ from \cite{MNS}. The group $G$ acts on $G/B$ by left multiplication. Hence, we have a Weyl group action on $H_T^*(G/B)$, which is denoted by $w^L$ for any $w\in W$. 

Define
$$\delta_i=\frac{1}{\alpha_i}(\id-s_i^L),\qquad \T_i^L=s_i^L-\delta_i,\qquad \T_i^{L,\vee} = s_i^L+\delta_i,$$
i.e.,
\begin{equation}\label{equ:leftoper}
    \T_i^L = \frac{-1}{\alpha_i}\id+\frac{\alpha_i+1}{\alpha_i}s_i^L,\qquad 
\T_i^{L,\vee} = \frac{1}{\alpha_i}\id+\frac{\alpha_i-1}{\alpha_i}s_i^L.
\end{equation}
Then we have (\cite[Theorem 4.3]{MNS})
\begin{equation}\label{eq:TCSMSSM}
 \T_i^L(\SSM(\mathring{\Sigma}_{w}))=\SSM(\mathring{\Sigma}_{s_iw}),\qquad 
\T_i^{L,\vee}(\CSM(\mathring{\Sigma}^{u}))=\CSM(\mathring{\Sigma}^{s_iu}).
\end{equation}
For the CSM classes of the Richardson cells, we have the following recursion formula.
\begin{Th}\label{thm:csmrec}
For $u,w\in W$ and $i\in I$, we have
\begin{align}\label{th4.1}
s_i^L(\CSM(\mathring{R}_{u,w}))
+ \alpha_i\cdot s_i^L(\CSM(\mathring{R}_{s_iu,w}))
= \CSM(\mathring{R}_{u,w})
+ \alpha_i\cdot\CSM(\mathring{R}_{u,s_iw}),
\end{align}
where whenever $u\not\leq w$, the summand $\CSM(\mathring{R}_{u,w})$ is understood to be zero. 
\end{Th}
\begin{proof}
Let $\varphi,\psi\in H_T^*(G/B)$. From the definition of $\T_i^L$ and $\T_i^{L,\vee}$, we have
$$s_i^L(\varphi\cdot \psi)+\alpha_i\cdot s_i^L(\varphi\cdot \T_i^{L,\vee}(\psi))
=\varphi\cdot\psi+\alpha_i\cdot(\T_i^L(\varphi)\cdot\psi).$$
Applying this formula to $\varphi=\CSM(\mathring{\Sigma}_{w})$, $\psi=\SSM(\mathring{\Sigma}^{u})$, and using   \eqref{eq:CSMR=CSMSSM} and \eqref{eq:TCSMSSM}, we get the desired formula. 
\end{proof}
\begin{Rmk}
    Using the right Demazure--Lusztig operators from \cite{MNS}, it is possible to prove another recursion for $\CSM(\mathring{R}_{u,w})$ via the same method.
\end{Rmk}

\subsection{Chern classes of open projected Richardson varieties}
Let $P$ be a parabolic subgroup of $G$ containing $B$. Recall that $\pi:G/B\rightarrow G/P$ is the natural projection. For any $u\leq w\in W^P$, the open projected Richardson variety is $\mathring{\Pi}_{u,w}=\pi(\mathring{R}_{u,w})$. Since $\pi|_{\mathring{\Sigma}_{w}}$ is injective, we get
\[\CSM(\mathring{\Pi}_{u,w})=\pi_*(\CSM(\mathring{R}_{u,w})).\]

On the other hand, for any $u,w\in W$,
$$\pi_*\big(\CSM(\mathring{R}_{u,w})\big)
=\pi_*\big(\CSM(\mathbbm{1}_{\mathring{R}_{u,w}})\big)
=\CSM\big(\pi_*(\mathbbm{1}_{\mathring{R}_{u,w}})\big).$$
By Lemma \ref{lem:WPinvof1uw}, for $u,w\in W$ and $v\in W_P$,
\begin{equation}\label{eq:projCSM}
\pi_*\big(\CSM({\mathring{R}_{u,w}})\big)
=\pi_*\big(\CSM({\mathring{R}_{uv,wv}})\big).
\end{equation}

By the same reason as in the previous section, we can define the operators $s_i^L$ on $H_T^*(G/P)$. Moreover, since the projection $\pi$ commutes with the $G$-action, $s_i^L$ commutes with the pushforward $\pi_*$. Applying $\pi_*$ to the equation \eqref{th4.1} in Theorem \ref{thm:csmrec}, we get the following Corollary. 
\begin{Coro}\label{coro:projrecu}
For $u,w\in W$ and $i\in I$, we have
$$s_i^L\big(\pi_*\big(\CSM(\mathring{R}_{u,w})\big)\big)
+ \alpha_i\cdot s_i^L\big(\pi_*\big(\CSM(\mathring{R}_{s_iu,w})\big)\big)
= \pi_*\big(\CSM(\mathring{R}_{u,w})\big)
+ \alpha_i\cdot\pi_*\big(\CSM(\mathring{R}_{u,s_iw})\big),$$
where whenever $u\not\leq_P w$, the term $\CSM(\mathring{R}_{u,w})$ is understood as zero. 
\end{Coro}

Now, let us rewrite the recursion in terms of the extended affine Weyl groups. 
Let $\lambda$ be a dominant cocharacter whose stabilizer is $W_P$. Recall that $f^{\lambda}_{u,w}=ut_\lambda w^{-1}$ as defined in   \eqref{eq:fuw}. Notice that $f^{\lambda}_{u,w}=f^{\lambda}_{uv,wv}$ for any $v\in W_P$.
Let us denote 
\begin{equation}\label{eq:defofBandB+}
\mathcal{B}=\{f^{\lambda}_{u,w}\mid u\leq_P w\}\quad \text{ and }\quad
\mathcal{B}^+=\{f^{\lambda}_{u,w}\mid u,w\in W\}.
\end{equation}
Note that for any $u, w\in W$, there exists a $v\in W_P$, such that $wv\in W^P$, and $f^{\lambda}_{u,w}=f^{\lambda}_{uv,wv}$. Hence, $\mathcal{B}^+=\{f^{\lambda}_{u,w}\mid u\in W, w\in W^P\}$.

We parameterize open projected Richardson varieties using $\mathcal{B}^+$ by denoting
$$\mathring{\Pi}_f=\mathring{\Pi}_{u,w},\qquad \text{ where } f:=f^{\lambda}_{u,w} \text{ for } u\in W, w\in W^P.$$
Note that when $f\notin \mathcal{B}$, $\mathring{\Pi}_f=\varnothing$ by Proposition \ref{prop:charofPBru}. 
Thanks to  \eqref{eq:projCSM}, 
$$\pi_*\big(\CSM(\mathring{R}_{u,w})\big) =\CSM(\mathring{\Pi}_f),\qquad 
\text{ where } f:=f^{\lambda}_{u,w} \text{ for } u,w\in W..$$
Hence, Corollary \ref{coro:projrecu} can be written as follows.
\begin{Coro}\label{coro:recloc}
For $f=f^\lambda_{u,w}\in \mathcal{B}^+$ and $i\in I$, we have
\begin{align}\label{coro4.3}
s_i^L(\CSM(\mathring{\Pi}_{f}))
+ \alpha_i\cdot s_i^L(\CSM(\mathring{\Pi}_{s_if}))
= \CSM(\mathring{\Pi}_{f})
+ \alpha_i\cdot\CSM(\mathring{\Pi}_{fs_i}).
\end{align}
\end{Coro}

The torus fixed points $(G/P)^T$ are indexed by $W/W_P\simeq W\lambda$, the Weyl group orbit of $\lambda$.
For any $\mu\in W\lambda$, let 
$$\cdot|_\mu: H_T^*(G/P)\longrightarrow H_T^*(\pt)$$
denote the localization to the fixed point $\mu$. 
By the localization theorem,
$$H_T^*(G/P)\longrightarrow H_T^*(\pt)^{\bigoplus W\lambda},\qquad \gamma\longrightarrow (\gamma|_\mu)_{\mu\in W\lambda}$$
is an injective map. From the definition of the operator $s_i^L$, we get for any $\gamma\in H_T^*(G/P)$,
\[s_i^L(\gamma)|_\mu = s_i(\gamma|_{s_i\mu}),\]
where the right-hand side denotes the usual Weyl group action on $H_T^*(\pt)=\Sym \mathfrak{t}^*$.

\subsection{Characterization of Chern classes}
First of all, we have the following recursion for the CSM classes of   open projected Richardson varieties.
\begin{Th}\label{themorem4.1}
Assume we are given 
$\{\gamma_{f,\mu}\in H_T^*(\pt): f\in \mathcal{B}^+,\mu\in W\lambda\}$
such that 
\begin{gather*}
\label{eq:CSMPRchar1}
\gamma_{ut_\lambda,\mu}=\delta_{u,\id}\delta_{\mu,\lambda}
\prod\nolimits_{\alpha\in R^+\setminus R_P^+}(-\alpha),
\qquad \forall u\in W,\mu\in W\lambda.
\\
\label{eq:CSMPRchar2}
s_i(\gamma_{f,s_i\mu})
+\alpha_i\cdot s_i(\gamma_{s_if,s_i\mu})
= \gamma_{f,\mu}
+\alpha_i\cdot \gamma_{fs_i,\mu},
\qquad \forall f\in \mathcal{B}^+,\mu\in W\lambda, i\in I.
\end{gather*}
Then for any $f\in \mathcal{B}^+$ and $\mu\in W\lambda$,
$$\gamma_{f,\mu}=\CSM(\mathring{\Pi}_f)|_\mu.$$
\end{Th}
\begin{proof}
Let us define
$$\{\gamma_{u,w,\mu}:u,w\in W,\mu\in W\lambda\}$$
by setting $\gamma_{u,w,\mu}= \gamma_{f,\mu}$ where $f=f^\lambda_{u,w}$. 
We will show by induction on $w\in W$ that 
\begin{equation}\label{eq:uwmuCSM}
\gamma_{u,w,\mu}=\pi_*\big(\CSM(\mathring{R}_{u,w})\big)|_{\mu},
\qquad \forall u\in W,\mu\in W\lambda.
\end{equation}
If $w=\id$, then $\mathring{R}_{u,w}=\Sigma_{\id}$ if $u=\id$, and empty otherwise. Thus, 
\[\pi_*\big(\CSM(\mathring{R}_{u,w})\big)|_{\mu}=\delta_{u,\id}[\Sigma_{\id}]|_{\mu}=\delta_{u,\id}\delta_{\mu,\lambda}
\prod\nolimits_{\alpha\in R^+\setminus R_P^+}(-\alpha)=\gamma_{u,\id,\mu}.\]

Localizing both sides of \eqref{coro4.3} in Corollary \ref{coro:recloc} to the fixed point $\mu\in G/P$, we get
$$
s_i(\CSM(\mathring{\Pi}_f)|_{s_i\mu})
+\alpha_i\cdot s_i(\CSM(\mathring{\Pi}_{s_if})|_{s_i\mu})
= \CSM(\mathring{\Pi}_f)|_{\mu}
+\alpha_i\cdot \CSM(\mathring{\Pi}_{fs_i})|_{\mu}. 
$$
Assume  \eqref{eq:uwmuCSM} holds for $w$, then the above equation and the second equation in the Theorem imply that  \eqref{eq:uwmuCSM} also holds for $s_iw$. This finishes the proof.
\end{proof}

Recall that  the SM classes are defined as  
\[\SSM(\mathring{\Pi}_f)=\frac{\CSM(\mathring{\Pi}_f)}{c^T(T(G/P))}.\] We get the following recursion for the SM classes.
\begin{Coro}\label{coro:charofSSM}
Assume we are given 
$\{\gamma_{f,\mu}\in H_T^*(\pt)_{\loc}: f\in \mathcal{B}^+,\mu\in W\lambda\}$
such that 
\begin{gather}
\label{eq:SSMPRchar1}
\gamma_{ut_\lambda,\mu}=\delta_{u,\id}\delta_{\mu,\lambda}
\prod\nolimits_{\alpha\in R^+\setminus R_P^+}
\frac{-\alpha}{1-\alpha},
\qquad \forall u\in W,\mu\in W\lambda.
\\
\label{eq:SSMPRchar2}
s_i(\gamma_{f,s_i\mu})
+\alpha_i\cdot s_i(\gamma_{s_if,s_i\mu})
= \gamma_{f,\mu}
+\alpha_i\cdot \gamma_{fs_i,\mu},
\qquad \forall u,w\in W,\mu\in W\lambda.
\end{gather}
Then for any $f\in \mathcal{B}^+$ and $\mu\in W\lambda$, we have 
$$\gamma_{f,\mu}=\SSM(\mathring{\Pi}_f)|_\mu.$$
\end{Coro}
\begin{proof}
    Since the tangent bundle $T(G/P)$ is $G$-equivariant, it is fixed by the action $s_i^L$. Hence,
    \[s_i^L(\SSM(\mathring{\Pi}_f))|_\mu=s_i(\SSM(\mathring{\Pi}_f))|_{s_i\mu})=\frac{s_i(\CSM(\mathring{\Pi}_f)|_{s_i\mu})}{c^T(T(G/P))|_\mu}.\]
    Then the Corollary follows from this and   Theorem \ref{themorem4.1}.
\end{proof}

\section{Recursion of Chern classes, affine part}\label{affinepart}

\subsection{CSM/SM classes of affine flag variety}
Let us first recall some properties of the CSM/SM classes of the Schubert cells in the affine flag variety $\Fl_G$. 

Let $I_{\aff}:=I\sqcup\{0\}$ be the vertices of the affine Dynkin diagram. By \cite{Ku02}, $\Fl_G$ can be realized as a Kac--Moody flag variety $\hat{G}/\hat{B}$, which is an ind-finite ind-scheme with a stratification by the finite-dimensional Schubert cells. The Weyl group of the Kac--Moody group $\hat{G}$ is $\widehat{W}$. Hence, the operator $T_i^L$ for $i\in I_{\aff}$ in  \eqref{equ:leftoper} can also be constructed for $H_*^T(\Fl_G)$, with $\alpha_0=-\theta$
as we only consider the small torus instead of the affine torus in $\hat{G}$, where $\theta$ is the highest root. 
On the other hand, there is another Weyl group $\widehat{W}$ action on $H_*^T(\Fl_G)$, see \cite[Definition 5.8]{KK86}. To distinguish the operator $w^L$ from Section \ref{sec:csmRichard}, we use $w^R$ to denote this action. For any $i\in I_{\aff}$, we use $\partial_i$ to denote the BGG operator \cite{BGG}, and let \[T_i^R:=\partial_i-s_i^R\] be the Hecke operator, see \cite{AM16, MNS}. Then both  sets of  operators $\{T_i^L\mid i\in I_{\aff}\}$ and $\{T_i^R\mid i\in I_{\aff}\}$ satisfy the relations in $\widehat{W}$, and $T_i^RT_j^L=T_j^LT_i^R$.

Recall that for any $f\in \widehat{W}$, $\mathring{\Sigma}_{f}$ denotes the Schubert cell. The CSM class $\CSM(\mathring{\Sigma}_{f})\in H_*^T(\Fl_G)$ is well defined as $\Fl_G$ is an ind-scheme, and they form a basis for the localized equivariant homology group. Since the Schubert variety $\Sigma_f:=\overline{\mathring{\Sigma}_{f}}$ also has a Bott--Samelson resolution as in the finite type case (see \cite{Ku02}), the proof of \cite[Theorem 1.1]{AM16} also works for the affine flag variety $\Fl_G$. Hence, we have the following formula
\[T_i^R(\CSM(\mathring{\Sigma}_{f}))=\CSM(\mathring{\Sigma}_{fs_i}),\]
where $i\in I_{\aff}$. Moreover, the proof in \cite[Theorem 4.3]{MNS} only depends on the above formula and a computation for $\mathbb{P}^1$, we get
\[T_i^L(\CSM(\mathring{\Sigma}_{f}))=\CSM(\mathring{\Sigma}_{s_if}).\]

Recall that the $T$-equivariant cohomology of the affine flag variety $H_T^{*}(\Fl_G)$ can be identified with $\Hom_{H^*_T(\pt)}(H^T_{*}(\Fl_G),H_T^{*}(\pt))$, see \cite[Chapter 4]{LLMSSZ14}. Hence, there is a perfect pairing 
\[\langle -,-\rangle: H^T_{*}(\Fl_G) \times H_T^{*}(\Fl_G) \rightarrow H_T^*(\pt).\]
We define the \emph{SM class of opposite Schubert cells} $\SSM(\mathring{\Sigma}^{f})\in H_T^{*}(\Fl_G)$ to be the dual basis of the CSM classes:
\begin{equation}\label{equ:defssm}
    \langle \CSM(\mathring{\Sigma} _f), \SSM(\mathring{\Sigma}^{g})\rangle=\delta_{f,g}.
\end{equation}
\begin{Rmk}
    To be more precise, we need to consider the thick affine flag variety $\tilde{\Fl}_G$, and the SM class is defined in the cohomology ring $H_T^*(\tilde{\Fl}_G)$, which is isomorphic to $\Hom_{H^*_T(\pt)}(H^T_{*}(\Fl_G),H_T^{*}(\pt))$, see \cite[Proposition 3.46 in Chapter 4]{LLMSSZ14}.
\end{Rmk}

Let \[\T_i^{L,\vee} = \frac{1}{\alpha_i}\id+\frac{\alpha_i-1}{\alpha_i}s_i^L,\quad \text{and}\quad T_i^{R,\vee}:=\partial_i+s_i^R.\] Then for any $\gamma_1\in H_*^T(\Fl_G)$ and $\gamma_2\in H_T^*(\Fl_G)$, we have (see \cite{AMSS23, MNS})
\[\langle T_i^{R}(\gamma_1), \gamma_2\rangle = \langle\gamma_1, T_i^{R,\vee}(\gamma_2)\rangle, \text{\quad and \quad} \langle T_i^{L}(\gamma_1), \gamma_2\rangle = s_i\cdot \langle\gamma_1, T_i^{L,\vee}(\gamma_2)\rangle.\]
Combining all the above equations, we get
\begin{equation*}\label{equ:LRcsm}
    T_i^{L,\vee} (\SSM(\mathring{\Sigma}^{f}))=\SSM(\mathring{\Sigma}^{s_if}),\qquad
T_i^{R,\vee} (\SSM(\mathring{\Sigma}^{f}))=\SSM(\mathring{\Sigma}^{fs_i}).
\end{equation*}

For any $f\in \widehat{W}$, we use $[e_f]\in H_*^T(\Fl_G)$ to denote the class of the fixed point corresponding to $f$. Then we have the following formula (see \cite{MNS}):
\begin{equation*}\label{equ:actonfixed}
    T_i^{R}([e_f])=\frac{1+f\alpha_i}{f\alpha_i}[e_{fs_i}]-\frac{1}{f\alpha_i}[e_f].
\end{equation*}

\subsection{Recursions of affine SM classes}
Recall that $H_T^*(\Fl_G)$ is defined to be the dual of $H_*^T(\Fl_G)$, there is a well-defined localization at fixed points  as follows (see \cite[Chapter 4]{LLMSSZ14})
\[\gamma|_f:=\langle [e_f],\gamma\rangle,\]
where $f\in \widehat{W}$ and $\gamma\in H_T^*(\Fl_G)$.
Then we have the following recursion for the localization of the SM classes. Recall that $\alpha_0=-\theta$ with $\theta$ being the highest root.
\begin{Prop}\label{prop:recursion}
Let $f,\t\in \widehat{W}$. 
For any $i\in I_{\aff}$, we have 
\begin{align}
\label{eq:recurofaffineSSML}
    (\alpha_i+1)
    \SSM(\mathring{\Sigma}^{f})|_{s_i\t}
    &=s_i\big(
    \SSM(\mathring{\Sigma}^{f})|_{\t}\big)
    +\alpha_i\cdot s_i\big(\SSM(\mathring{\Sigma}^{s_if})|_{\t}
    \big),
\\
\label{eq:recurofaffineSSMR}
    (\t(\alpha_i)+1)
    \SSM(\mathring{\Sigma}^{f})|_{\t s_i}
    & =
    \SSM(\mathring{\Sigma}^{f})|_{\t}
    +\t(\alpha_i)\cdot \SSM(\mathring{\Sigma}^{fs_i})|_{\t}
    .
\end{align}
\end{Prop}
\begin{proof}
    The first one follows from the definition of $T_i^{L,\vee}$, $T_i^{L,\vee} (\SSM(\mathring{\Sigma}^{f}))=\SSM(\mathring{\Sigma}^{s_if})$, and the fact that $s_i^L(\gamma)|_\t=s_i(\gamma|_{s_i\t})$ for any $\gamma\in H_T^*(\Fl_G)$. For the second one, we have
    \begin{align*}
        \SSM(\mathring{\Sigma}^{fs_i})|_\t&=\langle [e_\t],\SSM(\mathring{\Sigma}^{fs_i})\rangle\\
        &=\langle [e_\t], T_i^{R,\vee}(\SSM(\mathring{\Sigma}^{f}))\rangle\\
        &=\langle T_i^{R}([e_\t]), \SSM(\mathring{\Sigma}^{f})\rangle\\
        &=\frac{1+\t\alpha_i}{\t\alpha_i}\SSM(\mathring{\Sigma}^{f})|_{\t s_i}-\frac{1}{\t\alpha_i}\SSM(\mathring{\Sigma}^{f})|_\t.
    \end{align*}
    This finishes the proof.
\end{proof}

Recall that the action of $t_\mu\in \widehat{W}$ on the lattice $X^*(T)\oplus \mathbb{Z}\delta$ is given by the following formula
\[t_\mu(\lambda+k\delta)=\lambda+(k-\langle\lambda,\mu\rangle)\delta,\]
where $\lambda\in X^*(T)$ is a character of the maximal torus $T$, and $\delta$ is the imaginary root for the corresponding affine Kac--Moody algebra. Since we are considering the small torus $T$, $\delta$ is zero in $H_T^*(\pt)$. Hence, for any $\mu\in X_*(T)$, and $\lambda\in X^*(T)$, $t_\mu(\lambda)=\lambda$. Therefore, 
for any $i\in I$ and $\mu\in X_*(T)$, 
by substituting $\t=t_{s_i\mu}$ in \eqref{eq:recurofaffineSSML} and $\t=t_\mu$ in 
\eqref{eq:recurofaffineSSMR}, we have 
\begin{equation}\label{equ:affrec}
s_i\big(
    \SSM(\mathring{\Sigma}^{f})|_{t_{s_i\mu}}\big)
+\alpha_i\cdot s_i\big(\SSM(\mathring{\Sigma}^{s_if})|_{t_{s_i\mu}}
    \big)
=
\SSM(\mathring{\Sigma}^{f})|_{t_\mu}
+\alpha_i\cdot \SSM(\mathring{\Sigma}^{fs_i})|_{t_\mu}.
\end{equation}
This equality takes the form of \eqref{eq:SSMPRchar2}, and will be used in the proof of Theorem \ref{thm:pR=affine}.

\section{Comparison between the Chern classes}\label{comparison}

In this section, we will combine the results in the previous two sections to obtain a relationship between the SM classes of  open projected Richardson cells and the SM classes of  affine Schubert cells.

As before, let $\lambda$ be a dominant cocharacter and $P$ be a parabolic subgroup containing $B$ such that $W_P$ is the stabilizer of $\lambda$ in $W$. Let $f=f^\lambda_{u,w}=ut_\lambda w^{-1}\in \mathcal{B}^+$, where $u\in W$ and $w\in W^P$. Then we have the SM class $\SSM(\mathring{\Pi}_f)\in H_T^*(G/P)$ of the open projected Richardson variety. 
On the other hand, $f$ can be regarded as an element in the extended affine Weyl group $\widehat{W}$, and we have the SM class $\SSM(\mathring{\Sigma}^{f})\in H_T^*(\Fl_G)$.

For any $\mu\in X_*(T)$ and $\gamma\in H_T^*(\Gr_G)$, let $\gamma|_{t_\mu W}$ be the localization at the fixed point $t_\mu W$ in $\Gr_G$.
There is a pullback map (see \cite[Proposition 4.4 in Chapter 4]{LLMSSZ14})
\[r^*\colon H_T^*(\Fl_G)\rightarrow H_T^*(\Gr_G).\]
In terms of localization, this is defined as follows
\[r^*(\gamma)|_{t_{\mu}W}=\gamma|_{t_\mu},\]
where $\mu\in X_*(T)$. 

Recall that $i_\lambda:G/P\hookrightarrow\Gr_\lambda$ and $j_\lambda:\Gr_\lambda\hookrightarrow \Gr_G$ are inclusions. 
Let $q_\lambda^*:H_T^*(\Fl_G)\rightarrow H_T^*(\Gr_\lambda)$ be the composition of $j_\lambda^*\circ r^*$. 
The following is one of the main results of this paper.
\begin{Th}\label{thm:pR=aff}
Let $\mathcal{N}$ be the normal bundle of $G/P$ inside $\Gr_\lambda$. Then, for any $f=ut_{\lambda}w^{-1}\in \mathcal{B}^+$,  
$$i_{\lambda,*}\bigg(\SSM(\mathring{\Pi}_f)\cdot c^T(\mathcal{N})^{-1}\bigg)
= q_\lambda^*\bigg(\SSM(\mathring{\Sigma}^{f})\bigg)
\in H_T^*(\Gr_\lambda)_{\loc},
$$
where $c^T(\mathcal{N})$ is the $T$-equivariant total Chern class of $\mathcal{N}$.
\end{Th}

\begin{Rmk}
By taking the lowest degree terms, we obtain \cite[Theorem 5.8]{HL15}, the type $A$ case was  proved in \cite[Theorem 7.8]{KLS-Posi}.
\end{Rmk}

\begin{proof}
    By the localization theorem, we only need to check that both sides have the same localizations at the torus fixed points. The torus fixed points $(\Gr_\lambda)^T$ are $\{t_\mu W\mid \mu\in W\lambda\}$. Notice that the $T$-weights of the tangent space of $G/P$ at the identity point is $\{-\alpha\mid \langle\lambda,\alpha\rangle>0\}$. Hence, by Lemma \ref{lem:tanweGr}, 
    \[c^T(\mathcal{N})|_{\lambda }=\prod_{\langle\lambda,\alpha\rangle>0}(1-\alpha)^{\langle\lambda,\alpha\rangle-1}.\]
    Therefore, by the $G$-equivariance of $\mathcal{N}$,
    \[c^T(\mathcal{N})|_{\mu }=\prod_{\langle\mu,\alpha\rangle>0}(1-\alpha)^{\langle\mu,\alpha\rangle-1}\]
    for any $\mu\in W(\lambda)$.
    On the other hand, by the equivariant localization theorem,
    \begin{align*}
    i_{\lambda,*}\bigg(\SSM(\mathring{\Pi}_f)\cdot c^T(\mathcal{N})^{-1}\bigg)\bigg|_{t_\mu W}&=\SSM(\mathring{\Pi}_f)|_\mu  \cdot \frac{e^T(\mathcal{N})|_\mu}{c^T(\mathcal{N})|_\mu}\\
    &=\SSM(\mathring{\Pi}_f)|_\mu \prod_{\langle\alpha,\mu\rangle>0}\left(\frac{1-\alpha}{-\alpha}\right)^{-\langle \mu,\alpha \rangle+1},
    \end{align*}
    where $e^T(\mathcal{N})$ denotes the equivariant Euler class of $\mathcal{N}$.
    Then the theorem follows from the above equation and Theorem \ref{thm:pR=affine} below.
\end{proof}

\begin{Th}\label{thm:pR=affine}
For any $f\in \mathcal{B}^+$ and $\mu\in W\lambda$,  
$$\SSM(\mathring{\Pi}_f)|_{\mu}
= \SSM(\mathring{\Sigma}^{f})|_{t_\mu}
\prod_{\langle\alpha,\mu\rangle>0}
\left(\frac{1-\alpha}{-\alpha}
\right)^{\langle \mu,\alpha \rangle-1}\in H_T^*(\pt)_{\loc},
$$
where the product is over all the roots $\alpha$.
\end{Th}
\begin{proof}
    Let $\tilde{\gamma}_{f,\mu}$ denote the right-hand side of the equation, then we only need to check that $\tilde{\gamma}_{f,\mu}$ satisfies the two equations in Corollary \ref{coro:charofSSM}. Let us first check \eqref{eq:SSMPRchar1}. If $$\SSM(\mathring{\Sigma}^{ut_\lambda})|_{t_\mu}\neq 0, $$
    then $ut_\lambda\leq t_\mu$. By Theorem \ref{thm:affinecrit}, we must have $u=\id$. But $t_\lambda\leq t_\mu$ if and only if $\lambda=\mu$ since $\ell(t_\lambda)=\ell(t_\mu)$. 
    Let us compute $\SSM(\mathring{\Sigma}^{t_\lambda})|_{t_\lambda}$. Recall that for any $w\in \widehat{W}$, $[e_w]\in H_*^T(\Fl_G)$ denotes the class of the corresponding fixed point. Since the CSM class $\CSM(\mathring{\Sigma}_{t_\lambda})$ is supported on the closure $\overline{\mathring{\Sigma}_{t_\lambda}}$,  we have
    \[\CSM(\mathring{\Sigma}_{t_\lambda})=\sum_{g\in \widehat{W}, g\leq t_\lambda}a_g[e_g],\]
    for some coefficients $a_g\in H_T^*(\pt)_{\loc}$. Moreover, the leading coefficient $a_{t_\lambda}$, which is the equivariant multiplicity of $\CSM(\mathring{\Sigma}_{t_\lambda})$ at the torus fixed point $t_\lambda$, equals
    \[a_{t_\lambda}=\prod_{\chi}\frac{\chi+1}{\chi}=\prod_{\alpha>0}\bigg(\frac{1-\alpha}{-\alpha}\bigg)^{\langle\lambda, \alpha\rangle},\]
    Here the first equality follows from \cite[Theorem 6.5]{MNS22} with $\chi$ being the $T$-weights of the tangent space $T_{t_\lambda}(\mathcal{I}t_\lambda \mathcal{I}/\mathcal{I})$,  while the last one follows from Lemma \ref{lem:tangentwt}. Therefore, 
    \[[e_{t_\lambda}]=\sum_{g\in \widehat{W}, g\leq t_\lambda} b_g \CSM(\mathring{\Sigma}_g),\]
    for some coefficient $b_g\in H_T^*(\pt)_{\loc}$, with 
    \[b_{t_\lambda}=\frac{1}{a_{t_\lambda}}=\prod_{\alpha>0}\bigg(\frac{-\alpha}{1-\alpha}\bigg)^{\langle\lambda, \alpha\rangle}.\]
    Therefore, by the definition of  SM class in  \eqref{equ:defssm},
    \[\SSM(\mathring{\Sigma}^{t_\lambda})|_{t_\lambda}=\langle[e_{t_\lambda}], \SSM(\mathring{\Sigma}^{t_\lambda})\rangle=b_{t_\lambda}=\prod_{\alpha>0}\bigg(\frac{-\alpha}{1-\alpha}\bigg)^{\langle\lambda, \alpha\rangle}.\]
    Hence,  \eqref{eq:SSMPRchar1} is satisfied.

    Finally,  \eqref{eq:SSMPRchar2} follows from \eqref{equ:affrec} and the fact that
    \[\prod_{\langle\alpha,\mu\rangle>0}\left(\frac{1-\alpha}{-\alpha}\right)^{\langle \mu,\alpha \rangle-1}=s_i\bigg(\prod_{\langle\alpha,s_i\mu\rangle>0}\left(\frac{1-\alpha}{-\alpha}\right)^{\langle s_i\mu,\alpha \rangle-1}\bigg).
    \qedhere\]
\end{proof}

\begin{Eg}\label{eg:cominsculecase}
When $\lambda$ is minuscule, $G/P$ is said to be cominuscule. 
In this case $\Gr_\lambda\simeq G/P$ and $\mathcal{N}$ is trivial, see \cite{Zhu}. 
    We have 
    $$\SSM(\mathring{\Pi}_f)
    = q^*_\lambda\bigg(\SSM(\mathring{\Sigma}^{f})\bigg).
    $$
\end{Eg}

\begin{Eg}
When $u=w\in W^P$, 
the corresponding $f^\lambda_{u,w}=t_{\mu}$ for $\mu=w\lambda$. 
In this case the cell $\mathring{\Pi}_f$ is the fixed point corresponding to $w$. In particular, 
$$\CSM(\mathring{\Pi}_f)|_{\mu}=\delta_{\mu\lambda}\prod_{\beta\in R^+\setminus R_P^+}(-w\beta)
=\delta_{\mu\lambda}\prod_{\langle\alpha,\mu\rangle>0}(-\alpha).$$
As a result, 
$$\SSM(\mathring{\Pi}_f)|_{\mu}=\delta_{\mu\lambda}\prod_{\beta\in R^+\setminus R_P^+}(-w\beta)
=\delta_{\mu\lambda}\prod_{\langle\alpha,\mu\rangle>0}\frac{-\alpha}{1-\alpha}.$$
On the other hand, 
$$\SSM(\mathring{\Sigma}^{f})|_{t_\mu}
=\delta_{\mu\lambda}\prod_{\langle\alpha,\mu\rangle>0}
\left(\frac{-\alpha}{1-\alpha}\right)^{\langle\mu,\alpha\rangle}.$$
This verifies Theorem \ref{thm:pR=affine} when $f=t_\mu$.
\end{Eg}

\section{Combinatorial formula in type $A$}\label{combinpart}

In this section, we aim to provide  a combinatorial formula for $\SSM(\mathring{\Pi}_f)$ in the case of type $A$ via a string diagram inside $\mathbb{R}^2$.
In the type $A$ case, one can naturally identify
$$X_*(T)=\mathbb{Z}\mathbf{e}_1\oplus \cdots \oplus \mathbb{Z}\mathbf{e}_n\supset 
\mathbb{Z}(\mathbf{e}_1-\mathbf{e}_2)\oplus \cdots \oplus 
\mathbb{Z}(\mathbf{e}_{n-1}-\mathbf{e}_n)
=Q^\vee.$$
The extended affine Weyl group $\widehat{W}$ can be realized as the group of $n$-periodic affine permutations
$$\tilde{S}_n=\left\{\text{bijections   $f\colon  \mathbb{Z}\to \mathbb{Z}$ such that
$f(i+n)=f(i)+n$}\right\}.$$
For a cocharacter $\lambda=(\lambda_1,\ldots,\lambda_n)\in X_*(T)$ and $w\in S_n$,   $wt_{\lambda}\in \tilde{S}_n$ is the affine permutation   determined by 
$$wt_{\lambda}(i)=w(i)+\lambda_i\cdot n,\qquad 1\leq i\leq n.$$

To derive our combinatorial formula for $\SSM(\mathring{\Pi}_f)$, we need a localization  formula for $\SSM(\mathring{\Sigma}^f)$ as given in Theorem  \ref{thm:diamBilley}.

\subsection{Localization formula for $\SSM(\mathring{\Sigma}^f)$}\label{BGYU}

Following the classical convention of Schubert calculus, we identify
$H_T^*(\pt) = \mathbb{Q}[y_1,\ldots,y_n]$ where
$y_i\in X^*(T)$ such that $\langle y_i,\mathbf{e}_j\rangle=-\delta_{ij}$.
In particular, we have 
$$\alpha_i = -y_{i}+y_{i+1}\quad   (1\leq i\leq n-1), \quad  \alpha_0 = -y_n+y_1.$$

We shall represent any element $\t\in \tilde{S}_n$ by an $n$-periodic string diagram $\mathcal{D}_\t$. 
Intuitively, an $n$-periodic string diagram is a diagram $\mathcal{D}$ of upward strings inside $\mathbb{R}^2$, whose endpoints  are the lattice points  $(i,0)$ 
 and $(i,1)$ with $i\in \mathbb{Z}$,  
  such that it satisfies 
$\mathcal{D}+(n,0)=\mathcal{D}$ as well as the following conditions:
\begin{enumerate}
    \item all strings are smooth with tangent direction in $[0^\circ,180^\circ)$ at each point; 
    
    \item the intersection of any three strings must be empty;
    
    \item the tangent direction at the intersection point 
    of two strings    
    are different.
\end{enumerate}
That is, the following configurations of strings are banned
$$
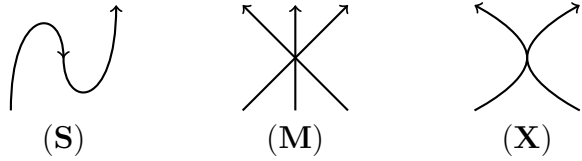
\begin{matrix}
\begin{tikzpicture}[scale=0.7]
\draw [thick,->] (0,0) 
.. controls (0,2) and (1,2).. 
(1,1)
.. controls (1,0) and (2,0).. 
node[pos=0,sloped,allow upside down]{\midarrow} 
(2,2);
\end{tikzpicture}\\
\mathbf{(S)}
\end{matrix}\qquad\qquad
\begin{matrix}
\begin{tikzpicture}[scale=0.7]
\draw [thick,->] (0,0) -- (2,2);
\draw [thick,->] (1,0) -- (1,2);
\draw [thick,->] (2,0) -- (0,2);
\end{tikzpicture}\\
\mathbf{(M)}
\end{matrix}\qquad\qquad
\begin{matrix}
\begin{tikzpicture}[scale=0.7]
\draw [thick,rounded corners] (0,0) 
.. controls (0,0) and (1,0.5).. 
(1,1);
\draw [thick,rounded corners,->] (1,1) 
.. controls (1,1.5) and (0,2).. 
(0,2);
\draw [thick,rounded corners] (2,0) 
.. controls (2,0) and (1,0.5).. 
(1,1);
\draw [thick,rounded corners,->] (1,1) 
.. controls (1,1.5) and (2,2).. 
(2,2);
\end{tikzpicture}\\
\mathbf{(X)}
\end{matrix}$$

For  $\t\in \tilde{S}_n$, a diagram $\mathcal{D}_\t$ is obtained by drawing a string connecting the endpoints  $(\t^{-1}(i),0)$ and $(i,1)$ for each $1\leq i\leq n$, and translating this local configuration of $n$ strings horizontally such that the resulting diagram is $n$-periodic. Note that such diagrams   are not unique. However, as will be seen later, we shall concern the weight generating function  of $\mathcal{D}_\t$, which is independent of the choice of $\mathcal{D}_\t$. 
In the following example, we illustrate a string  diagram 
$$\cdots\quad \begin{matrix}
\def\mytikzdiag{%
\draw[thick] (1,0) .. controls (1,1) and (3,1).. 
    node[pos=0.5,sloped,allow upside down]{\midarrow} 
    (4,2); 
\draw[thick] (2,0) .. controls (2,1) and (1,1).. 
    node[pos=0.55,sloped,allow upside down]{\midarrow} 
    (1,2); 
\draw[thick] (3,0) .. controls (3,0.7) and (5,1).. 
    node[pos=0.5,sloped,allow upside down]{\midarrow} 
    (5.5,1.2); 
\draw[thick] (0.5,1.2) .. controls (1,1.4) and (2,1.6).. 
    node[pos=0.7,sloped,allow upside down]{\midarrow} 
    (2,2); 
\draw[thick] (4,0) .. controls (5,1) and (5,1).. 
    node[pos=0.7,sloped,allow upside down]{\midarrow} 
    (3,2); 
\draw[thick] (5,0) .. controls (2,0.8) and (3,0.8).. 
    node[pos=0.7,sloped,allow upside down]{\midarrow} 
    (5,2); 
    }
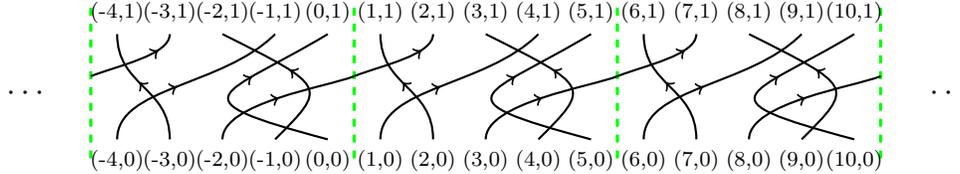
\begin{tikzpicture}[scale=0.7]
\draw [dashed,very thick,green] (-4.5,2.5)--(-4.5,-0.5);
\draw [dashed,very thick,green] (0.5,2.5)--(0.5,-0.5);
\draw [dashed,very thick,green] (5.5,2.5)--(5.5,-0.5);
\draw [dashed,very thick,green] (10.5,2.5)--(10.5,-0.5);
\begin{scope}[xshift=-5cm]
   \mytikzdiag
\end{scope}
\begin{scope}[xshift=0cm]
   \mytikzdiag
\end{scope}
\begin{scope}[xshift=5cm]
   \mytikzdiag
\end{scope}
\node at (1,0)[below]{\tiny(1,0)}; \node at (1,2)[above]{\tiny(1,1)};
\node at (2,0)[below]{\tiny(2,0)}; \node at (2,2)[above]{\tiny(2,1)};
\node at (3,0)[below]{\tiny(3,0)}; \node at (3,2)[above]{\tiny(3,1)};
\node at (4,0)[below]{\tiny(4,0)}; \node at (4,2)[above]{\tiny(4,1)};
\node at (5,0)[below]{\tiny(5,0)}; \node at (5,2)[above]{\tiny(5,1)};
\node at (0,0)[below]{\tiny(0,0)}; \node at (0,2)[above]{\tiny(0,1)};
\node at (-1,0)[below]{\tiny(-1,0)};\node at (-1,2)[above]{\tiny(-1,1)};
\node at (-2,0)[below]{\tiny(-2,0)};\node at (-2,2)[above]{\tiny(-2,1)};
\node at (-3,0)[below]{\tiny(-3,0)};\node at (-3,2)[above]{\tiny(-3,1)};
\node at (-4,0)[below]{\tiny(-4,0)};\node at (-4,2)[above]{\tiny(-4,1)};
\node at (6,0)[below]{\tiny(6,0)};\node at (6,2)[above]{\tiny(6,1)};
\node at (7,0)[below]{\tiny(7,0)};\node at (7,2)[above]{\tiny(7,1)};
\node at (8,0)[below]{\tiny(8,0)};\node at (8,2)[above]{\tiny(8,1)};
\node at (9,0)[below]{\tiny(9,0)};\node at (9,2)[above]{\tiny(9,1)};
\node at (10,0)[below]{\tiny(10,0)};\node at (10,2)[above]{\tiny(10,1)};
\end{tikzpicture}
\end{matrix}\quad\cdots$$
for the affine permutation $\t\in \tilde{S}_5$ with 
$$\t(1)=4,\quad \t(2)=1,\quad \t(3)=7,\quad \t(4)=3,\quad \t(5)=5.$$

To state the localization formula, we define two colorings on a string diagram $\mathcal{D}$. 
An $n$-periodic coloring of {\it endpoints} of $\mathcal{D}$ is a map $\beta\colon P\mapsto \beta(P)\in \mathbb{Z}$, which assigns each endpoint $P$ with an integer such that $\beta(P+(n,0))=\beta(P)+n$. 
Notice that the intersection points of strings cut the strings into pieces. An $n$-periodic coloring $\kappa$ on the {\it pieces} assigns each piece an integer such that 
\begin{itemize}
    \item for any piece $p$, $\kappa(p+(n,0))=\kappa(p)+n$;  
    \item if an intersection point $P$ of two strings has surrounding pieces  colored as follows
    $$
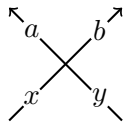
\begin{matrix}
\begin{tikzpicture}[scale=1.5] 
\draw [thick,->] (0,0) 
-- 
node[pos=0.2]{\midlabel{$x$}} 
node[pos=0.8]{\midlabel{$b$}} 
(1,1); 
\draw [thick,->] (1,0) 
--
node[pos=0.2]{\midlabel{$y$}} 
node[pos=0.8]{\midlabel{$a$}} 
 (0,1); 
\end{tikzpicture}
\end{matrix}$$
then we require that either $a=x\neq b=y$ or $a=y\neq x=b$. 
\end{itemize}
We say that a coloring $\kappa$ of pieces is  \emph{compatible} with a coloring  $\beta$ of endpoints if for any piece $p$ and each endpoint $P$ of $p$,  we have $\kappa(p)=\beta(P)$.
For example, in the following figure
$$\cdots\qquad
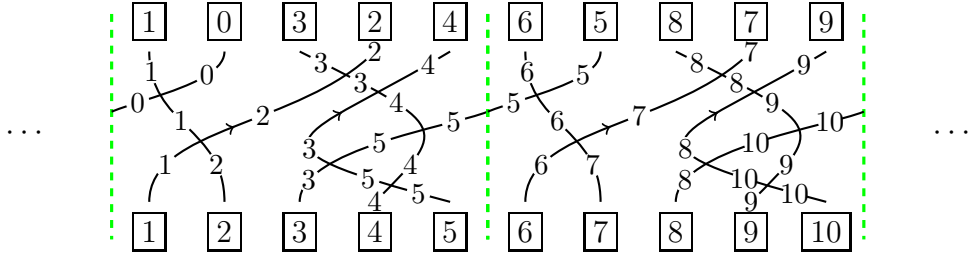
\begin{matrix}
\begin{tikzpicture}
\begin{scope}[xshift=0cm]
\draw [dashed,very thick,green] (0.5,2.5)--(0.5,-0.5);
\draw [dashed,very thick,green] (5.5,2.5)--(5.5,-0.5);
\draw[thick] (1,0) .. controls (1,1) and (3,1).. 
    node[pos=0.5,sloped,allow upside down]{\midarrow} 
    node[pos=0.2]{\midlabel{1}} 
    node[pos=0.6]{\midlabel{2}} 
    node[pos=1.0]{\midlabel{2}} 
    (4,2); 
\draw[thick] (2,0) .. controls (2,1) and (1,1).. 
    node[pos=0.55,sloped,allow upside down]{\midarrow} 
    node[pos=0.2]{\midlabel{2}} 
    node[pos=0.55]{\midlabel{1}} 
    node[pos=0.9]{\midlabel{1}} 
    (1,2); 
\draw[thick] (3,0) .. controls (3,0.7) and (5,1).. 
    node[pos=0.5,sloped,allow upside down]{\midarrow} 
    node[pos=0.15]{\midlabel{3}} 
    node[pos=0.5]{\midlabel{5}} 
    node[pos=0.8]{\midlabel{5}} 
    (5.5,1.2); 
\draw[thick] (0.5,1.2) .. controls (1,1.4) and (2,1.6).. 
    node[pos=0.7,sloped,allow upside down]{\midarrow} 
    node[pos=0.2]{\midlabel{0}} 
    node[pos=0.7]{\midlabel{0}} 
    (2,2); 
\draw[thick] (4,0) .. controls (5,1) and (5,1).. 
    node[pos=0.7,sloped,allow upside down]{\midarrow} 
    node[pos=0.0]{\midlabel{4}} 
    node[pos=0.2]{\midlabel{4}} 
    node[pos=0.7]{\midlabel{4}} 
    node[pos=0.84]{\midlabel{3}} 
    node[pos=0.95]{\midlabel{3}} 
    (3,2); 
\draw[thick] (5,0) .. controls (2,0.8) and (3,0.8).. 
    node[pos=0.7,sloped,allow upside down]{\midarrow} 
    node[pos=0.05]{\midlabel{5}} 
    node[pos=0.15]{\midlabel{5}} 
    node[pos=0.4]{\midlabel{3}} 
    node[pos=0.95]{\midlabel{4}} 
    (5,2); 
\node at (1,0)[below]{\fbox{1}}; \node at (1,2)[above]{\fbox{1}};
\node at (2,0)[below]{\fbox{2}}; \node at (2,2)[above]{\fbox{0}};
\node at (3,0)[below]{\fbox{3}}; \node at (3,2)[above]{\fbox{3}};
\node at (4,0)[below]{\fbox{4}}; \node at (4,2)[above]{\fbox{2}};
\node at (5,0)[below]{\fbox{5}}; \node at (5,2)[above]{\fbox{4}};
\end{scope}
\begin{scope}[xshift=5cm]
\draw [dashed,very thick,green] (0.5,2.5)--(0.5,-0.5);
\draw [dashed,very thick,green] (5.5,2.5)--(5.5,-0.5);
\draw[thick] (1,0) .. controls (1,1) and (3,1).. 
    node[pos=0.5,sloped,allow upside down]{\midarrow} 
    node[pos=0.2]{\midlabel{6}} 
    node[pos=0.6]{\midlabel{7}} 
    node[pos=1.0]{\midlabel{7}} 
    (4,2); 
\draw[thick] (2,0) .. controls (2,1) and (1,1).. 
    node[pos=0.55,sloped,allow upside down]{\midarrow} 
    node[pos=0.2]{\midlabel{7}} 
    node[pos=0.55]{\midlabel{6}} 
    node[pos=0.9]{\midlabel{6}} 
    (1,2); 
\draw[thick] (3,0) .. controls (3,0.7) and (5,1).. 
    node[pos=0.5,sloped,allow upside down]{\midarrow} 
    node[pos=0.15]{\midlabel{8}} 
    node[pos=0.5]{\midlabel{10}} 
    node[pos=0.8]{\midlabel{10}} 
    (5.5,1.2); 
\draw[thick] (0.5,1.2) .. controls (1,1.4) and (2,1.6).. 
    node[pos=0.7,sloped,allow upside down]{\midarrow} 
    node[pos=0.2]{\midlabel{5}} 
    node[pos=0.7]{\midlabel{5}} 
    (2,2); 
\draw[thick] (4,0) .. controls (5,1) and (5,1).. 
    node[pos=0.7,sloped,allow upside down]{\midarrow} 
    node[pos=0.0]{\midlabel{9}} 
    node[pos=0.2]{\midlabel{9}} 
    node[pos=0.7]{\midlabel{9}} 
    node[pos=0.84]{\midlabel{8}} 
    node[pos=0.95]{\midlabel{8}} 
    (3,2); 
\draw[thick] (5,0) .. controls (2,0.8) and (3,0.8).. 
    node[pos=0.7,sloped,allow upside down]{\midarrow} 
    node[pos=0.05]{\midlabel{10}} 
    node[pos=0.15]{\midlabel{10}} 
    node[pos=0.4]{\midlabel{8}} 
    node[pos=0.95]{\midlabel{9}} 
    (5,2); 
\node at (1,0)[below]{\fbox{6}}; \node at (1,2)[above]{\fbox{6}};
\node at (2,0)[below]{\fbox{7}}; \node at (2,2)[above]{\fbox{5}};
\node at (3,0)[below]{\fbox{8}}; \node at (3,2)[above]{\fbox{8}};
\node at (4,0)[below]{\fbox{9}}; \node at (4,2)[above]{\fbox{7}};
\node at (5,0)[below]{\fbox{10}}; \node at (5,2)[above]{\fbox{9}};
\end{scope}
\end{tikzpicture}
\end{matrix}\qquad \cdots
$$
we see that $\kappa$ is compatible with $\beta$, where we set $\beta(i,0)=i$ and 
$$\beta(1,1)=1,\quad \beta(2,1)=0,\quad \beta(3,1)=3,\quad \beta(4,1)=2,\quad \beta(5,1)=4,\quad \text{etc}.$$
In the remaining of this section, we always assume that $\kappa$ is compatible with $\beta$.

By assigning   a weight to  each string (not each piece), we can  define the weight of each intersection point $P$ in the following manner: 
$$\wt\left(\begin{matrix}
\begin{tikzpicture}[scale=1.5] 
\draw [thick,->] (0,0) node [left]{$u$}
-- 
node[pos=0.2]{\midlabel{$x$}} 
node[pos=0.8]{\midlabel{$b$}} 
(1,1) node [left]{$u$}; 
\draw [thick,->] (1,0) node [left]{$v$}
--
node[pos=0.2]{\midlabel{$y$}} 
node[pos=0.8]{\midlabel{$a$}} 
 (0,1) node [left]{$v$}; 
\end{tikzpicture}
\end{matrix}\right)
=\frac{1}{1+u-v}\begin{cases}
1,&a=x\neq b=y,\\
u-v,& a=y\neq x=b.
\end{cases}$$
Here $u, v$ are the weights of two intersecting strings. 
The weight $\wt(\mathcal{D},\beta,\kappa)$ of a string diagram $\mathcal{D}$ is defined as the product of weights of all intersection points inside one periodicity  $(\epsilon,\epsilon+n]\times \mathbb{R}$ for any generic $\epsilon\in \mathbb{R}$. 
Moreover, we define the weight $\wt(\mathcal{D},\beta)$ to be  the sum of  
$\wt(\mathcal{D},\beta,\kappa)$ with $\kappa$ running through colorings compatible with $\beta$.

The following property is well known, see for example \cite[Proposition 2.1]{KJII}. 
\begin{Th}\label{theorem7.1}
We have the following weight-preserving local moves
$$
\begin{matrix}
    \begin{matrix}
    \begin{tikzpicture}[scale=0.7]
    \draw [thick,rounded corners,->] (0,0) -- (2,1.4)-- (2,2);
    \draw [thick,rounded corners,->] (1,0) -- (0,0.6)-- (0,1.4) -- (1,2);  
    \draw [thick,rounded corners,->] (2,0) -- (2,0.6)-- (0,2);  
    \end{tikzpicture}
    \end{matrix}\quad =\quad
    \begin{matrix}
    \begin{tikzpicture}[scale=0.7]
    \draw [thick,rounded corners,->] (0,0) --(0,0.6)-- (2,2) ;
    \draw [thick,rounded corners,->] (1,0) -- (2,0.6)-- (2,1.4) -- (1,2) ;  
    \draw [thick,rounded corners,->] (2,0) -- (0,1.4)-- (0,2) ;  
    \end{tikzpicture}
    \end{matrix}\\
\text{Yang--Baxter equation }{\mathbf{(YBE)}}
\end{matrix}
\quad
\begin{matrix}
    \begin{matrix}
    \begin{tikzpicture}[scale=0.7]
    \draw [thick,rounded corners,->] (1,0) -- (0,1) -- (1,2); 
    \draw [thick,rounded corners,->] (0,0) -- (1,1) -- (0,2); 
    \end{tikzpicture}
    \end{matrix}
    \quad=\quad
    \begin{matrix}
    \begin{tikzpicture}[scale=0.7]
    \draw [thick,rounded corners,->] (1,0) -- (1,1) -- (1,2); 
    \draw [thick,rounded corners,->] (0,0) -- (0,1) -- (0,2); 
    \end{tikzpicture}
    \end{matrix}\\
\text{unitary equation }{\mathbf{(UE)}}
\end{matrix}
\quad
\begin{matrix}
    \begin{matrix}
    \begin{tikzpicture}[scale=0.7] 
    \draw [thick,rounded corners,->] (1,0) --node[pos=0.7,left]{$u$} (0,2); 
    \draw [thick,rounded corners,->] (0,0) --node[pos=0.7,right]{$u$} (1,2); 
    \end{tikzpicture}
    \end{matrix}
    =
    \begin{matrix}
    \begin{tikzpicture}[scale=0.7]
    \draw [thick,rounded corners,->] (1,0) --node[pos=0.7,right]{$u$} (1,2); 
    \draw [thick,rounded corners,->] (0,0) --node[pos=0.7,left]{$u$} (0,2); 
    \end{tikzpicture}
    \end{matrix}\\
\text{normalization }{\mathbf{(Nm)}}
\end{matrix}
$$
\end{Th}
\begin{Rmk}
    Notice that the $\mathbf{(YBE)}$ could be viewed as a version of braid relation in type $A$. 
\end{Rmk}

For $\t\in \tilde{S}_n$, the first two equations in Theorem \ref{theorem7.1}  imply that 
$\wt(\mathcal{D}_\t,\beta)$ is well defined, that is, it is independent of  the choice of the string diagram $\mathcal{D}_\t$. 
In fact, since the configuration  in  $\mathbf{(S)}$ is not allowed, one can deform any string,  with two endpoints fixed, via moving horizontally. 
If we choose the deformation generically, then, during the movement, 
we will only meet the local configurations  $\mathbf{(M)}$ and $\mathbf{(X)}$.
So the string diagrams just before and just after have the same weights by $\mathbf{(YBE)}$ and $\mathbf{(UE)}$ respectively. This means that $\wt(\mathcal{D}_\t,\beta)$ only depends on how the endpoints are connected.

We now assign the string connecting $(\t^{-1}(i),0)$ 
 and $(i,1)$ with the weight $y_i$, where $1\leq i\leq n$. 
For $f\in \tilde{S}_n$, we define $\beta_f$ as the coloring such that $\beta_f(i,0)=i$ and $\beta_f(i,1)=f^{-1}(i)$. 
Diagrammatically, it looks like
$$
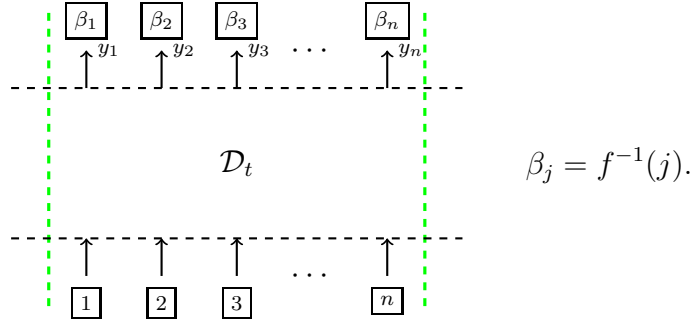
\begin{matrix}
\begin{tikzpicture}[scale=1]
\draw [dashed,very thick,green] (0.5,3.5)--(0.5,-0.5);
\draw [dashed,very thick,green] (5.5,3.5)--(5.5,-0.5);
\draw [dashed,thick] (0,0.5)--(6,0.5);
\draw [dashed,thick] (0,2.5)--(6,2.5);
\node at (3,1.5){$\mathcal{D}_\t$};
\draw [thick,->](1,2.5)--(1,3)node[above]{\fbox{\tiny $\beta_1$}}
node[right]{$\scriptstyle y_1$};
\draw [thick,->](2,2.5)--(2,3)node[above]{\fbox{\tiny $\beta_2$}}
node[right]{$\scriptstyle y_2$};
\draw [thick,->](3,2.5)--(3,3)node[above]{\fbox{\tiny $\beta_3$}}
node[right]{$\scriptstyle y_3$};
\node at (4,3){$\cdots$};
\draw [thick,->](5,2.5)--(5,3)node[above]{\fbox{\tiny $\beta_n$}}
node[right]{$\scriptstyle y_n$};
\draw [thick,->](1,0)node[below]{\fbox{\tiny $1$}}--(1,0.5);
\draw [thick,->](2,0)node[below]{\fbox{\tiny $2$}}--(2,0.5);
\draw [thick,->](3,0)node[below]{\fbox{\tiny $3$}}--(3,0.5);
\node at (4,0){$\cdots$};
\draw [thick,->](5,0)node[below]{\fbox{\tiny $n$}}--(5,0.5);
\end{tikzpicture}
\end{matrix}\qquad 
\beta_j=f^{-1}(j).$$

The following localization formula should be known  to experts. We include a brief argument here since we could not find a proof in the literature. 

\begin{Th}\label{thm:diamBilley}
For $f,\t\in \tilde{S}_n$, we have
$$\wt(\mathcal{D}_\t,\beta_f)=\SSM(\mathring{\Sigma}^f)|_\t.$$
\end{Th}
\begin{proof}
When $\ell(\t)=0$, it is easily checked that  $$\wt(\mathcal{D}_{\t},\beta_f)=\delta_{f,\t}
=\SSM(\mathring{\Sigma}^f)|_\t.$$
For any $\t$ and $i\in I\cup \{0\}$, consider the following string diagram for $s_i\t$:
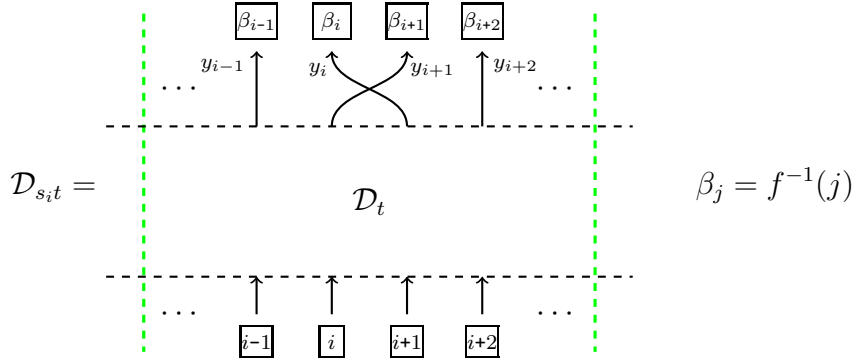
$$\mathcal{D}_{s_i\t}=\begin{matrix}
\begin{tikzpicture}[scale=1]
\draw [dashed,very thick,green] (0.5,2.5)--(0.5,-2);
\draw [dashed,very thick,green] (6.5,2.5)--(6.5,-2);
\draw [thick,<-] (2,2) 
node[above]{\fbox{\tiny $\!\!\beta_{i\texttt{-}1}$\!\!}}
-- node[pos=0.2,left]{$\scriptstyle y_{i-1}$}
(2,1); 
\draw [thick,<-] (3,2) 
node[above]{\fbox{\tiny $\beta_i$}}
.. controls (3,1.5) and (4,1.5).. 
node[pos=0.2,left]{$\scriptstyle y_{i}$}(4,1); 
\draw [thick,<-] (4,2) 
node[above]{\fbox{\tiny $\!\!\beta_{i\texttt{+}1}$\!\!}}
.. controls (4,1.5) and (3,1.5).. 
node[pos=0.2,right]{$\scriptstyle y_{i+1}$}(3,1); 
\draw [thick,<-] (5,2) 
node[above]{\fbox{\tiny $\!\!\beta_{i\texttt{+}2}\!\!$}}
-- node[pos=0.2,right]{$\scriptstyle y_{i+2}$}(5,1); 
\node at (1,1.5) {$\cdots$};
\node at (6,1.5) {$\cdots$};
\draw [dashed,thick] (0,1) -- (7,1);
\node at (3.5,0){$\mathcal{D}_{\t}$}; 
\draw [dashed,thick] (0,-1) --(7,-1);
\draw [thick,<-] (2,-1)--(2,-1.5)
node[below]{\fbox{\tiny $\!\!{i\texttt{-}1}$\!\!}};
\draw [thick,<-] (3,-1)--(3,-1.5)
node[below]{\fbox{\tiny ${i}$}};
\draw [thick,<-] (4,-1)--(4,-1.5)
node[below]{\fbox{\tiny $\!\!{i\texttt{+}1}\!\!$}};
\draw [thick,<-] (5,-1)--(5,-1.5)
node[below]{\fbox{\tiny $\!\!{i\texttt{+}2}\!\!$}};
\node at (1,-1.5) {$\cdots$};
\node at (6,-1.5) {$\cdots$};
\end{tikzpicture}
\end{matrix}\qquad \beta_j=f^{-1}(j)$$
Removing the intersection point gives a string diagram $\mathcal{D}_\t$. Notice that there are two choices for the colors of the lower two pieces attached to this intersection point. This yields the following equality
\begin{align*}
\wt(\mathcal{D}_{s_i\t},\beta_{f})
& =
\frac{1}{1+y_{i+1}-y_i}\left(
s_i\big(\wt(\mathcal{D}_{\t},\beta_{f})\big)
+(y_{i+1}-y_i)
s_i\big(\wt(\mathcal{D}_{\t},\beta_{s_if})\big)\right)\\
& = 
\frac{1}{1+\alpha_i}\left(
s_i\big(\wt(\mathcal{D}_{\t},\beta_{f})\big)
+\alpha_i
s_i\big(\wt(\mathcal{D}_{\t},\beta_{s_if})\big)\right),
\end{align*}
which agrees with the recurrence  \eqref{eq:recurofaffineSSML} in Proposition \ref{prop:recursion}. 
So the theorem follows by induction. 
\end{proof}

\subsection{Positroid varieties}
We are now in a position to  construct  the symmetric rational function, denoted $\tilde{F}_f$, which represents the class  $\SSM(\mathring{\Pi}_f)$ of the open positroid variety $\mathring{\Pi}_f$ in the Grassmannian. 
We choose the minuscule cocharacter
$$\lambda = \mathbf{e}_1+\cdots+\mathbf{e}_k\in X_*(T). $$
Then $G/P$ is the Grassmannian  $\Gr_k(\mathbb{C}^n)$. 
In this case, open projected Richardson varieties  are known as {open positroid varieties}, which are indexed by bounded affine  permutations \cite{KLS-Posi}: 
$$\mathcal{B}
=\left\{
\begin{array}{c}
\text{bijections $f\colon\mathbb{Z}\to \mathbb{Z}$}
\end{array}:
\begin{array}{c}
f(i+n)=f(i)+n\\[1ex]
\frac{1}{n}\sum_{i=1}^n(f(i)-i)=k\\[1ex]
i\leq f(i)\leq i+n
\end{array}
\right\}.$$

Combining 
Example \ref{eg:cominsculecase} and Theorem \ref{thm:diamBilley}, we obtain the following localization formula for $\SSM(\mathring{\Pi}_f)$. 

\begin{Coro}\label{coro:SSMPi=wtD}
We have
$\SSM(\mathring{\Pi}_f)|_{\lambda} =\wt(\mathcal{D}_{t_{\lambda}},\beta_{f})$. 
\end{Coro}
\begin{Rmk}
It could be possible to compute the left-hand side by finding explicit resolutions, similar to the approaches of \cite[Proposition 6.2, Corollary 6.3]{KRW20}.  
\end{Rmk}

To define $\tilde{F}_f$, let us consider another type of string diagrams: the grid $\Delta$ in $\mathbb{Z}^2$ including $k$ horizontal lines.   Color the endpoints  of vertical lines  using $\beta_f$ as in Subsection \ref{BGYU}. The horizontal and vertical lines are assigned with weights $x_1,\ldots, x_k$ and $y_1,\ldots, y_n$ as illustrated  below:
$$(\Delta,\beta_f)=
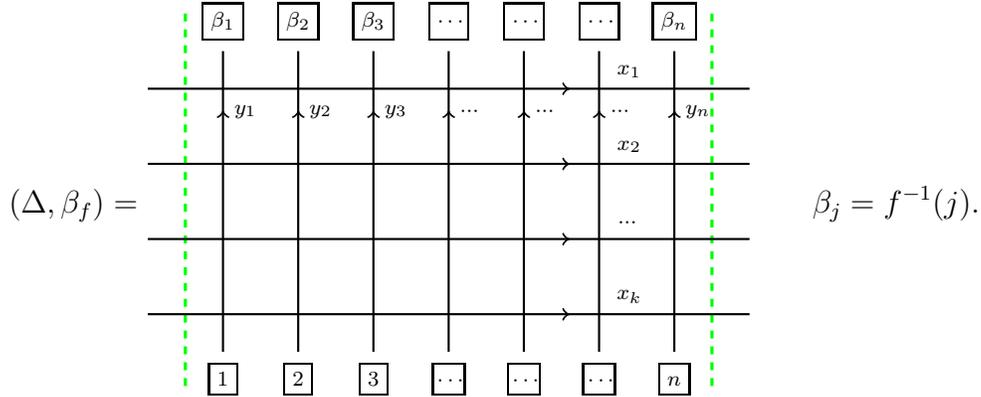
\begin{matrix}
\begin{tikzpicture}[scale=1]
\draw [dashed,very thick,green] (0.5,4.5)--(0.5,-0.5);
\draw [dashed,very thick,green] (7.5,4.5)--(7.5,-0.5);
\draw [thick] (0,0.5)--
    node[pos=0.8,above]{$\scriptstyle x_k$}
    node[pos=0.7,sloped,allow upside down]{\midarrow} (8,0.5);
\draw [thick] (0,1.5)--
    node[pos=0.8,above]{$\scriptstyle\cdots$}
    node[pos=0.7,sloped,allow upside down]{\midarrow} (8,1.5);
\draw [thick] (0,2.5)--
    node[pos=0.8,above]{$\scriptstyle x_2$}
    node[pos=0.7,sloped,allow upside down]{\midarrow} (8,2.5);
\draw [thick] (0,3.5)--
    node[pos=0.8,above]{$\scriptstyle x_1$}
    node[pos=0.7,sloped,allow upside down]{\midarrow} (8,3.5);
\draw [thick] (1,0) node[below]{\fbox{\tiny 1}}--
    node[pos=0.8,right]{$\scriptstyle y_1$}
    node[pos=0.8,sloped,allow upside down]{\midarrow}(1,4)
    node[above]{\fbox{\tiny$\beta_1$}};
\draw [thick] (2,0) node[below]{\fbox{\tiny 2}}--
    node[pos=0.8,right]{$\scriptstyle y_2$}
    node[pos=0.8,sloped,allow upside down]{\midarrow}(2,4)
    node[above]{\fbox{\tiny$\beta_2$}};
\draw [thick] (3,0) node[below]{\fbox{\tiny 3}}--
    node[pos=0.8,right]{$\scriptstyle y_3$}
    node[pos=0.8,sloped,allow upside down]{\midarrow}(3,4)
    node[above]{\fbox{\tiny$\beta_3$}};
\draw [thick] (4,0) node[below]{\fbox{\tiny$\!\!\vphantom{1}\cdots\!\!$}}--
    node[pos=0.8,right]{$\scriptstyle \cdots$}
    node[pos=0.8,sloped,allow upside down]{\midarrow}(4,4)
    node[above]{\fbox{\tiny$\!\vphantom{\beta_1}\cdots\!$}};
\draw [thick] (5,0) node[below]{\fbox{\tiny$\!\!\vphantom{1}\cdots\!\!$}}--
    node[pos=0.8,right]{$\scriptstyle \cdots$}
    node[pos=0.8,sloped,allow upside down]{\midarrow}(5,4)
    node[above]{\fbox{\tiny$\!\vphantom{\beta_1}\cdots\!$}};
\draw [thick] (6,0) node[below]{\fbox{\tiny$\!\!\vphantom{1}\cdots\!\!$}}--
    node[pos=0.8,right]{$\scriptstyle \cdots$}
    node[pos=0.8,sloped,allow upside down]{\midarrow}(6,4)
    node[above]{\fbox{\tiny$\!\vphantom{\beta_1}\cdots\!$}};
\draw [thick] (7,0) node[below]{\fbox{\tiny$\vphantom{1}n$}}--
    node[pos=0.8,right]{$\scriptstyle y_n$}
    node[pos=0.8,sloped,allow upside down]{\midarrow}(7,4)
    node[above]{\fbox{\tiny$\beta_n$}};
\end{tikzpicture}
\end{matrix}\qquad 
\beta_j=f^{-1}(j).$$
Similarly, we may define $n$-periodic colorings $\kappa$ of the segments connecting the intersection points. For each $\kappa$ which is compatible with $\beta_f$,  we accordingly define the weight $\wt(\Delta,\beta_f,\kappa)$ as the product of weights of all intersection points inside  one periodicity. Here the weight of an intersection point obeys the same rule as defined in 
Subsection \ref{BGYU}. 
Summing over all colorings $\kappa$ compatible with $\beta_f$, we obtain 
the weight generating function  $\wt(\Delta,\beta_f)$, which is the function  that we require:
$$\tilde{F}_f(x_1,\ldots,x_k;y_1,\ldots,y_n):=\wt(\Delta,\beta_f).$$

\begin{Th}
For $f\in \mathcal{B}$, the function $\tilde{F}_f$ is symmetric in $x$. 
\end{Th}

\begin{proof}
It suffices to show that  $\tilde{F}_f$ is symmetric if exchanging $x_i$ and $x_{i+1}$. 
This is illustrated through  the following diagrammatic calculation:
\begin{align*}
\begin{matrix}
\begin{tikzpicture}[scale=0.7]
\draw [dashed,very thick,green] (0.5,1.5)--(0.5,-0.5);
\draw [dashed,very thick,green] (5.5,1.5)--(5.5,-0.5);
\draw [thick] (0,0)--
    node[pos=0.5,below]{$\scriptstyle x_{i+1}$}
    node[pos=0.75,sloped,allow upside down]{\midarrow} (6,0);
\draw [thick] (0,1)--
    node[pos=0.5,above]{$\scriptstyle x_{i}$}
    node[pos=0.75,sloped,allow upside down]{\midarrow} (6,1);
\draw [thick,->] (1,-0.5)--(1,1.5); 
\draw [thick,->] (2,-0.5)--(2,1.5); 
\node at (3,0.5){$\cdots$};
\draw [thick,->] (4,-0.5)--(4,1.5); 
\draw [thick,->] (5,-0.5)--(5,1.5); 
\end{tikzpicture}
\end{matrix}
& =
\begin{matrix}
\begin{tikzpicture}[scale=0.7]
\draw [dashed,very thick,green] (0.5,1.5)--(0.5,-0.5);
\draw [dashed,very thick,green] (7,1.5)--(7,-0.5);
\draw [thick] (0.5,0)--
    node[pos=0.5,below]{$\scriptstyle x_{i+1}$}
    node[pos=0.85,sloped,allow upside down]{\midarrow} +(4.5,0);
\draw [thick] (0.5,1)--
    node[pos=0.5,above]{$\scriptstyle x_{i}$}
    node[pos=0.85,sloped,allow upside down]{\midarrow} +(4.5,0);
\draw [thick] (-0.5,1) .. controls +(.5,0) and +(-.5,0) .. +(1,-1);
\draw [thick] (-0.5,0) .. controls +(.5,0) and +(-.5,0) .. +(1,1);
\draw [thick,->] (0.75,-0.5)--+(0,2); 
\draw [thick,->] (1.75,-0.5)--+(0,2);
\node at (2.75,0.5){$\cdots$};
\draw [thick,->] (3.75,-0.5)--+(0,2); 
\draw [thick,->] (4.75,-0.5)--+(0,2); 
\draw [thick] (5,1) .. controls (5.5,1) and (5.5,0) .. (6,0);
\draw [thick] (5,0) .. controls (5.5,0) and (5.5,1) .. (6,1);
\draw [thick] (6,1) .. controls (6.5,1) and (6.5,0) .. (7,0);
\draw [thick] (6,0) .. controls (6.5,0) and (6.5,1) .. (7,1);
\draw [thick] (7,1)--(7.5,1);
\draw [thick] (7,0)--(7.5,0);
\end{tikzpicture}
\end{matrix}
    && \text{by \textbf{(UE)}}
\\&=
\begin{matrix}
\begin{tikzpicture}[scale=0.7]
\draw [dashed,very thick,green] (0.5,1.5)--(0.5,-0.5);
\draw [dashed,very thick,green] (7,1.5)--(7,-0.5);
\draw [thick] (0.5,0)--
    node[pos=0.6,below]{$\scriptstyle x_{i+1}$}
    node[pos=0.95,sloped,allow upside down]{\midarrow} +(4,0);
\draw [thick] (0.5,1)--
    node[pos=0.6,above]{$\scriptstyle x_{i}$}
    node[pos=0.95,sloped,allow upside down]{\midarrow} +(4,0);
\draw [thick] (5.5,0)--(6,0);
\draw [thick] (5.5,1)--(6,1);
\draw [thick] (-0.5,1) .. controls +(.5,0) and +(-.5,0) .. +(1,-1);
\draw [thick] (-0.5,0) .. controls +(.5,0) and +(-.5,0) .. +(1,1);
\draw [thick,->] (0.75,-0.5)--+(0,2); 
\draw [thick,->] (1.75,-0.5)--+(0,2);
\node at (2.75,0.5){$\cdots$};
\draw [thick,->] (3.75,-0.5)--+(0,2); 
\draw [thick,->] (5.75,-0.5)--+(0,2); 
\draw [thick] (4.5,1) .. controls +(0.5,0) and +(-0.5,0) .. +(1,-1);
\draw [thick] (4.5,0) .. controls +(0.5,0) and +(-0.5,0) .. +(1,1);
\draw [thick] (6,1) .. controls (6.5,1) and (6.5,0) .. (7,0);
\draw [thick] (6,0) .. controls (6.5,0) and (6.5,1) .. (7,1);
\draw [thick] (7,1)--(7.5,1);
\draw [thick] (7,0)--(7.5,0);
\end{tikzpicture}
\end{matrix}    && \text{by \textbf{(YBE)}}\\
& = \cdots  =\\&=
\begin{matrix}
\begin{tikzpicture}[scale=0.7]
\draw [dashed,very thick,green] (0.5,1.5)--(0.5,-0.5);
\draw [dashed,very thick,green] (7,1.5)--(7,-0.5);
\draw [thick] (1.5,0)--
    node[pos=0.5,below]{$\scriptstyle x_{i}$}
    node[pos=0.85,sloped,allow upside down]{\midarrow} +(4.5,0);
\draw [thick] (1.5,1)--
    node[pos=0.5,above]{$\scriptstyle x_{i+1}$}
    node[pos=0.85,sloped,allow upside down]{\midarrow} +(4.5,0);
\draw [thick,->] (1.75,-0.5)--+(0,2); 
\draw [thick,->] (2.75,-0.5)--+(0,2); 
\node at (3.75,0.5){$\cdots$};
\draw [thick,->] (4.75,-0.5)--+(0,2); 
\draw [thick,->] (5.75,-0.5)--+(0,2); 
\draw [thick] (0.5,1) .. controls +(.5,0) and +(-.5,0) .. +(1,-1);
\draw [thick] (0.5,0) .. controls +(.5,0) and +(-.5,0) .. +(1,1);
\draw [thick] (6,1) .. controls +(.5,0) and +(-.5,0) .. +(1,-1);
\draw [thick] (6,0) .. controls +(.5,0) and +(-.5,0) .. +(1,1);
\draw [thick] (7,1) .. controls +(.5,0) and +(-.5,0) .. +(1,-1);
\draw [thick] (7,0) .. controls +(.5,0) and +(-.5,0) .. +(1,1);
\draw [thick] (-.5,1) .. controls +(.5,0) and +(-.5,0) .. +(1,-1);
\draw [thick] (-.5,0) .. controls +(.5,0) and +(-.5,0) .. +(1,1);
\end{tikzpicture}
\end{matrix}
    && \text{by \textbf{(YBE)}}\\
& =
\begin{matrix}
\begin{tikzpicture}[scale=0.7]
\draw [dashed,very thick,green] (0.5,1.5)--(0.5,-0.5);
\draw [dashed,very thick,green] (5.5,1.5)--(5.5,-0.5);
\draw [thick] (0,0)--
    node[pos=0.5,below]{$\scriptstyle x_{i}$}
    node[pos=0.75,sloped,allow upside down]{\midarrow} (6,0);
\draw [thick] (0,1)--
    node[pos=0.5,above]{$\scriptstyle x_{i+1}$}
    node[pos=0.75,sloped,allow upside down]{\midarrow} (6,1);
\draw [thick,->] (1,-0.5)--(1,1.5); 
\draw [thick,->] (2,-0.5)--(2,1.5); 
\node at (3,0.5){$\cdots$};
\draw [thick,->] (4,-0.5)--(4,1.5); 
\draw [thick,->] (5,-0.5)--(5,1.5); 
\end{tikzpicture}
\end{matrix}
    && \text{by \textbf{(UE)}}\qedhere
\end{align*}
\end{proof}

\begin{Th}\label{thm:Ff=SSM}
For $f\in \mathcal{B}$, we have 
$$\tilde{F}_f=\SSM(\mathring{\Pi}_f)\in H_T^{*}(\Gr_k(\mathbb{C}^n))_{\loc}.$$
\end{Th}

\begin{proof}
It is enough  to check that 
$$\tilde{F}_f|_{\mu}=\SSM(\mathring{\Pi}_f)|_{\mu}$$
for any $\mu\in S_n\lambda$.
Let $a_1<\cdots<a_k$ be the indices such that $\mu_{a_i}=1$. Then 
$$\tilde{F}_f|_{\mu}=\tilde{F}_f(y_{a_1},\ldots,y_{a_k};y_1,\ldots,y_n)\in H_T^*(\pt).$$
We see that ($a=a_i$ in the following diagram)
\begin{align*}
\left.\left(
\begin{matrix}
\begin{tikzpicture}[scale=0.7]
\draw [dashed,very thick,green] (0.5,1.5)--(0.5,-0.5);
\draw [dashed,very thick,green] (5.5,1.5)--(5.5,-0.5);
\draw [thick] (0,0.5)--
    node[pos=0.8,below]{$\scriptstyle x_{i}$}
    node[pos=0.75,sloped,allow upside down]{\midarrow} (6,0.5);
\draw [thick,->] (3,-0.5)--(3,1.5)node[right]{$\scriptstyle y_a$}; 
\draw [thick,->] (4,-0.5)--(4,1.5)node[right]{$\scriptstyle y_{a\texttt{+}1}$}; 
\draw [thick,->] (2,-0.5)--(2,1.5)node[right]{$\scriptstyle\!\! y_{a\texttt{-}1}$}; 
\node at (1.2,-.2){$\cdots$};
\node at (4.8,-.2){$\cdots$};
\node at (1.2,1){$\cdots$};
\node at (4.8,1){$\cdots$};
\end{tikzpicture}
\end{matrix}\right)\right|_{x_i\mapsto y_a}
& = 
\begin{matrix}
\begin{tikzpicture}[scale=0.7]
\draw [dashed,very thick,green] (0.5,1.5)--(0.5,-0.5);
\draw [dashed,very thick,green] (5.5,1.5)--(5.5,-0.5);
\draw [thick] (0,0.5)--
    node[pos=0.8,below]{$\scriptstyle y_{a}$}
    node[pos=0.75,sloped,allow upside down]{\midarrow} (6,0.5);
\draw [thick,->] (3,-0.5)--(3,1.5)node[right]{$\scriptstyle y_a$}; 
\draw [thick,->] (4,-0.5)--(4,1.5)node[right]{$\scriptstyle y_{a\texttt{+}1}$}; 
\draw [thick,->] (2,-0.5)--(2,1.5)node[right]{$\scriptstyle\!\! y_{a\texttt{-}1}$}; 
\node at (1.2,-.2){$\cdots$};
\node at (4.8,-.2){$\cdots$};
\node at (1.2,1){$\cdots$};
\node at (4.8,1){$\cdots$};
\end{tikzpicture}
\end{matrix}
\\&=
\begin{matrix}
\begin{tikzpicture}[scale=0.7]
\draw [dashed,very thick,green] (0.5,1.5)--(0.5,-0.5);
\draw [dashed,very thick,green] (5.5,1.5)--(5.5,-0.5);
\draw [thick] (0,0.5)--(2,0.5);
\draw [thick] (4,0.5)--
    node[pos=0.2,sloped,allow upside down]{\midarrow}
    (6,0.5);
\draw [thick,->] (2,0.5).. controls (3,0.5) and (3,0.5) .. (3,1.5)
    node[right]{$\scriptstyle y_a$};
\draw [thick] (4,0.5).. controls (3,0.5) and (3,0.5) .. (3,-.5)
    node[right]{$\scriptstyle y_a$};
\draw [thick,->] (4,-0.5)--(4,1.5)node[right]{$\scriptstyle y_{a\texttt{+}1}$}; 
\draw [thick,->] (2,-0.5)--(2,1.5)node[right]{$\scriptstyle\!\! y_{a\texttt{-}1}$}; 
\node at (1.2,-.2){$\cdots$};
\node at (4.8,-.2){$\cdots$};
\node at (1.2,1){$\cdots$};
\node at (4.8,1){$\cdots$};
\end{tikzpicture}
\end{matrix}\qquad \text{by \textbf{(Nm)}}
\end{align*}
So, after the specialization $x_i\mapsto y_{a_i}$, the diagram $\Delta$ becomes a string diagram for $t_\mu$. 
By Corollary \ref{coro:SSMPi=wtD}, we have $\tilde{F}_f|_{\mu}=\SSM(\mathring{\Pi}_f)|_{\mu}$.
\end{proof}

For a concrete  example to illustrate the above proof, consider the case $(k,n)=(3,7)$ and $\mu=(1,0,1,0,0,1,0)$. Then $t_{\mu}(1)=8,t_{\mu}(2)=2,t_{\mu}(3)=10,t_{\mu}(4)=4,t_{\mu}(5)=5,t_{\mu}(6)=13,t_{\mu}(7)=7.$ By  \textbf{(Nm)}, we have
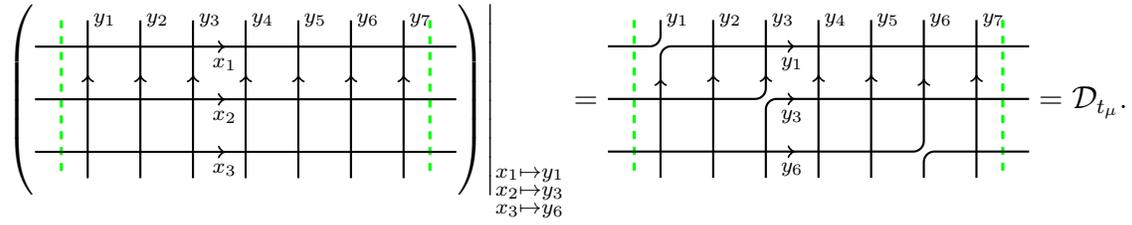
$$\left.
\left(
\begin{matrix}
\begin{tikzpicture}[scale=0.7]
\draw [dashed,very thick,green] (0.5,3)--(0.5,0);
\draw [dashed,very thick,green] (7.5,3)--(7.5,0);
\draw [thick] (0,0.5)--
    node[pos=0.45,sloped,allow upside down]{\midarrow}
    node[below,pos=0.45]{$\scriptstyle x_3$}
    (8,0.5);
\draw [thick] (0,1.5)--
    node[pos=0.45,sloped,allow upside down]{\midarrow}
    node[below,pos=0.45]{$\scriptstyle x_2$}
    (8,1.5);
\draw [thick] (0,2.5) --
    node[pos=0.45,sloped,allow upside down]{\midarrow}
    node[below,pos=0.45]{$\scriptstyle x_1$}
    (8,2.5);
\draw [thick] (1,0)--
    node[pos=0.65,sloped,allow upside down]{\midarrow}
    (1,3) node[right]{$\!\scriptstyle y_1$};
\draw [thick] (2,0)  --
    node[pos=0.65,sloped,allow upside down]{\midarrow}
    (2,3) node[right]{$\!\scriptstyle y_2$};
\draw [thick] (3,0)  --
    node[pos=0.65,sloped,allow upside down]{\midarrow}
    (3,3) node[right]{$\!\scriptstyle y_3$};
\draw [thick] (4,0)  --
    node[pos=0.65,sloped,allow upside down]{\midarrow}
    (4,3) node[right]{$\!\scriptstyle y_4$};
\draw [thick] (5,0)  --
    node[pos=0.65,sloped,allow upside down]{\midarrow}
    (5,3) node[right]{$\!\scriptstyle y_5$};
\draw [thick] (6,0)  --
    node[pos=0.65,sloped,allow upside down]{\midarrow}
    (6,3) node[right]{$\!\scriptstyle y_6$};
\draw [thick] (7,0)  --
    node[pos=0.65,sloped,allow upside down]{\midarrow}
    (7,3) node[right]{$\!\scriptstyle y_7$};
\end{tikzpicture}
\end{matrix}\right)\right|_{
\begin{subarray}{l}
x_1\mapsto y_1\\
x_2\mapsto y_3\\
x_3\mapsto y_6\\
\end{subarray}}
=
\begin{matrix}
\begin{tikzpicture}[scale=0.7]
\draw [dashed,very thick,green] (0.5,3)--(0.5,0);
\draw [dashed,very thick,green] (7.5,3)--(7.5,0);
\draw [thick,rounded corners] (0,0.5)--
    (6,0.5)-- 
    node[pos=0.55,sloped,allow upside down]{\midarrow}
    (6,3)
    node[right]{$\!\scriptstyle y_6$}; 
\draw [thick,rounded corners] (6,0)--
    (6,0.5)-- 
    (8,0.5); 
\draw [thick,rounded corners] (0,1.5)--
    (3,1.5)-- 
    node[pos=0.3,sloped,allow upside down]{\midarrow}
    (3,3)
    node[right]{$\!\scriptstyle y_3$}; 
\draw [thick,rounded corners] (3,0)--
    (3,1.5)-- 
    (8,1.5); 
\draw [thick,rounded corners] (0,2.5)--
    (1,2.5)-- 
    (1,3)
    node[right]{$\!\scriptstyle y_1$}; 
\draw [thick,rounded corners] (1,0)--
    node[pos=0.75,sloped,allow upside down]{\midarrow}
    (1,2.5)-- 
    (8,2.5); 
\draw [thick] (2,0)  --
    node[pos=0.65,sloped,allow upside down]{\midarrow}
    (2,3) node[right]{$\!\scriptstyle y_2$};
\draw [thick] (4,0)  --
    node[pos=0.65,sloped,allow upside down]{\midarrow}
    (4,3) node[right]{$\!\scriptstyle y_4$};
\draw [thick] (5,0)  --
    node[pos=0.65,sloped,allow upside down]{\midarrow}
    (5,3) node[right]{$\!\scriptstyle y_5$};
\draw [thick] (7,0)  --
    node[pos=0.65,sloped,allow upside down]{\midarrow}
    (7,3) node[right]{$\!\scriptstyle y_7$};
\draw (3.4,0.5) -- 
    node[pos=0.7,sloped,allow upside down]{\midarrow}
    node[below]{$\scriptstyle y_6$}
(3.6,0.5);
\draw (3.4,1.5) -- 
    node[pos=0.7,sloped,allow upside down]{\midarrow}
    node[below]{$\scriptstyle y_3$}
(3.6,1.5);
\draw (3.4,2.5) -- 
    node[pos=0.7,sloped,allow upside down]
    {\midarrow}
    node[below]{$\scriptstyle y_1$}
    (3.6,2.5);
\end{tikzpicture}
\end{matrix}=\mathcal{D}_{t_\mu}.$$

\subsection{Pipe dream model}
If we take the Poincar\'e dual of the diagram above, we will reach  a pipe dream model for $\tilde{F}_f$ as follows. 
Consider all possible $n$-periodic tiling on the square  grid $ \{1,\ldots,k\}\times \mathbb{Z}$ using tiles
$$\PD[1.5pc]{\B}\qquad \PD[1.5pc]{\X}$$
Here, as usual, $n$-periodicity means the tile at $(i,j)$ is the same as that at $(i+n,j)$. We emphasize that $i$ denotes  the row index, counted from top to bottom, and $j$ denotes the column index, counted from left to right. 
Let us denote by $\mathsf{PD}(f)$ the set of all such tilings with reading affine permutation $f$. 
Alternatively, each such tiling is obtained from a triple $(\Delta,\beta_f,\kappa)$ by reconnecting the edges around each vertex such that two edges are joined  if they  receive the same color from $\kappa$. 
For example, when $n=7$, $k=3$, 
the following tiling 
$$\PD{
\M{}\D
\M{\text{-}9}\M{\text{-}6}
\M{\text{-}\kern -0.1em 1\kern -0.15em0}
\M{\text{-}8}\M{\text{-}4}\M{\text{-}5}\M{0}
\D
\M{\text{-}2}\M{1}\M{\text{-}3}\M{\text{-}1}\M{3}\M{2}\M{7}
\D
\M{5}\M{8}\M{4}\M{6}\M{1\kern -0.1em0}\M{9}\M{1\kern -0.1em4}
\D
\\
\M{}\D
    \X\X\B\X\X\B\X
\D\X\X\B\X\X\B\X\D\X\X\B\X\X\B\X\D\\
\M{\cdots\qquad\qquad}\D
    \B\B\X\B\X\B\X
\D\B\B\X\B\X\B\X\D\B\B\X\B\X\B\X\D\M{\qquad\qquad\cdots}\\
\M{}\D
    \X\B\B\X\B\B\X
\D\X\B\B\X\B\B\X\D\X\B\B\X\B\B\X\D
\\
\M{}\D
\M{\text{-}6}\M{\text{-}5}\M{\text{-}4}\M{\text{-}3}\M{\text{-}2}\M{\text{-}1}\M{0}
\D\M{1}\M{2}\M{3}\M{4}\M{5}\M{6}\M{7}
\D\M{8}\M{9}\M{1\kern -0.1em0}\M{1\kern -0.1em1}\M{1\kern -0.1em2}\M{1\kern -0.1em3}\M{1\kern -0.1em4}\D\\
}$$
is a pipe dream  with  reading affine permutation given by
$$f(1)=2,\,\, f(2)=6,\,\, f(3)=5,\,\, f(4)=10,\,\, f(5)=8,\,\, f(6)=11,\,\, f(7)=7.$$
For  $\pi\in \mathsf{PD}(f)$, set
$$\wt(\pi)=\prod_{i=1}^k\prod_{j=1}^n
\frac{1}{1+x_i-y_j}
\begin{cases}
1, &\text{the $(i,j)$-position is $\PD{\B}$},\\
x_i-y_j,& \text{the $(i,j)$-position is $\PD{\X}$}.
\end{cases}$$
Then, in the language of pipe dreams, we see that 
$$\tilde{F}_f=\sum_{\pi\in \mathsf{PD}(f)}\wt(\pi).$$

\begin{Rmk}
When compared  with \cite{SZ2025}, it follows  that 
the lowest degree component of $\tilde{F}_f$ is the double affine Stanley symmetric function \cite{LLS21}. 
\end{Rmk}

\begin{Eg}Consider the case $\mathbb{P}^{n-1}\cong \Gr(1,n)$.
Let $f\in \mathcal{B}$. Define 
$$A=\{1\leq i\leq n: f(i)= i\}\subset [n].$$
Note that there is only one element in $\mathsf{PD}(f)$, that is, for $1\leq i\leq n$,
$$\text{the $(1,i)$ tile is }\begin{cases}
\PD{\B}, & i\notin A,\\
\PD{\X}, & i\in A.
\end{cases}$$
For example, when $A=\{2,4,5,6,9\}\subset [9]$, 
$$\PD{
\D\M{\text{-}1}\M{2}\M{1}\M{4}\M{5}\M{6}\M{3}\M{7}\M{9}\D\\
\D\B\X\B\X\X\X\B\B\X\D\\
\D\M{1}\M{2}\M{3}\M{4}\M{5}\M{6}\M{7}\M{8}\M{9}\D\\
}$$
is the only element in $\mathsf{PD}(f)$ with weight 
$$\frac{(x-y_2)(x-y_4)(x-y_5)(x-y_6)(x-y_9)}{(1+x-y_1)(1+x-y_2)\cdots (1+x-y_9)}.$$

Geometrically, the open positroid variety can be described as a torus orbit
$$\mathring{\Pi}_f=\big\{[x_1:\cdots:x_n]:
x_i=i\iff i\in A\big\}.$$
Note that in a smooth toric variety, the CSM class of a toric orbit is nothing but the fundamental class of its closure \cite[Section 5.3 Lemma]{Fulton}, so
$$\CSM(\mathring{\Pi}_f) = [\Pi_f]=\prod_{i\in A}(x-y_i)\in H_T^*(\mathbb{P}^{n-1}).$$
Thus 
$$\SSM(\mathring{\Pi}_f) = \prod_{i=1}^n
\frac{1}{1+x-y_i}
\begin{cases}
1, & i\notin A,\\
x-y_i, & i\in A.
\end{cases}$$
This agrees with our formula. 
\end{Eg}

\begin{Eg}
Let $f\in \tilde{S}_4$ be such that 
$$f(1)=2,\quad f(2)=5,\quad f(3)=4,\quad f(4)=7$$
Compute $\SSM(\mathring{\Pi}_f)\in H_T^*(\Gr(2,4))$. There are  six elements in $\mathsf{PD}(f)$:
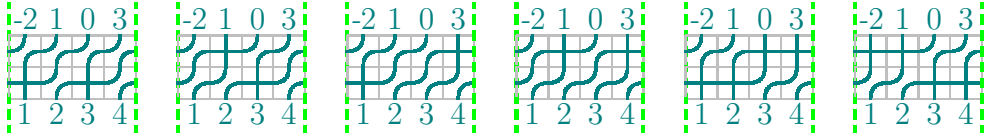
$$\PD{
\D\M{\text{-}2}\M{1}\M{0}\M{3}\D\\
\D\B\B\B\B\D\\
\D\X\B\X\B\D\\
\D\M{1}\M{2}\M{3}\M{4}\D\\
}\quad 
\PD{
\D\M{\text{-}2}\M{1}\M{0}\M{3}\D\\
\D\B\X\B\B\D\\
\D\B\B\X\B\D\\
\D\M{1}\M{2}\M{3}\M{4}\D\\
}\quad 
\PD{
\D\M{\text{-}2}\M{1}\M{0}\M{3}\D\\
\D\B\B\B\X\D\\
\D\X\B\B\B\D\\
\D\M{1}\M{2}\M{3}\M{4}\D\\
}\quad 
\PD{
\D\M{\text{-}2}\M{1}\M{0}\M{3}\D\\
\D\B\X\B\X\D\\
\D\B\B\B\B\D\\
\D\M{1}\M{2}\M{3}\M{4}\D\\
}\quad 
\PD{
\D\M{\text{-}2}\M{1}\M{0}\M{3}\D\\
\D\B\B\X\X\D\\
\D\X\X\B\B\D\\
\D\M{1}\M{2}\M{3}\M{4}\D\\
}\quad 
\PD{
\D\M{\text{-}2}\M{1}\M{0}\M{3}\D\\
\D\X\X\B\B\D\\
\D\B\B\X\X\D\\
\D\M{1}\M{2}\M{3}\M{4}\D\\
}$$
Thus 
$$\SSM(\mathring{\Pi}_f)=
\frac{1}{\displaystyle\prod_{i=1}^2\prod_{j=1}^4(1+x_i-y_j)}\left(
\begin{array}{c}
(x_2-y_1)(x_2-y_3)
+(x_1-y_2)(x_2-y_3)\\
+(x_1-y_4)(x_2-y_1)
+(x_1-y_2)(x_1-y_4)\\
+(x_1-y_3)(x_1-y_4)(x_2-y_1)(x_2-y_2)\\
+(x_1-y_1)(x_1-y_2)(x_2-y_3)(x_2-y_4)
\end{array}
\right).$$
\end{Eg}


\begin{thebibliography}{99}

\bibitem{AM16}
P. Aluffi and L. Mihalcea, {\em Chern–Schwartz–MacPherson classes for Schubert cells in flag manifolds}, Compos. Math. 152 (2016), 2603--2625. 

\bibitem{AMSS23} 
P. Aluffi, L. Mihalcea,  J. Sch\"urmann, and C. Su, {\em Shadows of characteristic cycles, Verma modules, and positivity of Chern--Schwartz--MacPherson classes of Schubert cells}, Duke Math. J. 172 (2023), 3257--3320.

\bibitem{BGG}
I.N. Bernstein, I.M. Gel'fand, and S.I. Gel'fand, 
{\em Schubert cells and cohomology of the spaces $G/P$}, Russ. Math. Surv. 28 (1973), 1--26.

\bibitem{BC}
S. Billey and I. Coskun, {\em Singularities of generalized Richardson varieties}, Comm. Algebra  40 (2012), 1466--1495.

\bibitem{BB}
A. Bj\"orner and F. Brenti, {\em Combinatorics of Coxeter groups}, Grad. Texts in Math., 231, Springer, New York, 2005.

\bibitem{BCMP18}
A.S. Buch, P.-E. Chaput, L. Mihalcea, and N. Perrin, {\em Projected Gromov--Witten varieties
in cominuscule spaces}, Proc. Amer. Math. Soc. 146 (2018), 3647–3660.

\bibitem{BKT03}
A. S. Buch, A. Kresch, and H. Tamvakis, {\em Gromov--Witten invariants on Grassmannians}, J. Amer. Math. Soc. 16 (2003),  901–915.  

\bibitem{BS}
J.P. Brasselet and M.H. Schwartz, {\em Sur les classes de Chern d’un ensemble analytique complexe (French), The Euler--Poincar\'e characteristic (French)}, Ast\'erisque, vol. 82, Soc. Math. France, Paris, 1981, pp. 93–147.

\bibitem{FGX}
N. Fan, P. Guo, and R. Xiong,
{\em Pieri and Murnaghan--Nakayama type rules for Chern classes of Schubert cells}, \href{https://arxiv.org/abs/2211.06802}{arXiv:2211.06802}, 2022.


\bibitem{Fulton}
W. Fulton, {\it Introduction to toric varieties}, Ann. of Math. Stud. 131, Princeton University Press, Princeton U.S.A., 1993.


\bibitem{GL19}
P. Galashin and T. Lam, {\em Positroid varieties and cluster algebras},   Ann. Sci. \'Ec. Norm. Sup\'er.  56 (2023),  859--884.

\bibitem{GL24}
P. Galashin and T. Lam, {\em
Positroids, knots, and $q,t$-Catalan numbers}, Duke Math. J. 173 (2024), 2117--2195. 


\bibitem{GY09}
K. R. Goodearl and M. Yakimov, {\em Poisson structures on aﬃne spaces and flag varieties. II}, Trans.
Amer. Math. Soc. 361 (2009),  5753--5780.

\bibitem{Grothen}
A. Grothendieck, 
{\it R\'ecoltes et Semailles}, available in English translation at 
\url{https://web.ma.utexas.edu/users/slaoui/notes/recoltes_et_semailles.pdf}

 
\bibitem{HL15}
X. He  and T. Lam, {\em
Projected Richardson varieties and affine Schubert varieties},
Ann. Inst. Fourier (Grenoble)   65 (2015),   2385--2412. 

\bibitem{IwaMat}
N. Iwahori and H. Matsumoto, {\em On some Bruhat decomposition and the structure of the Hecke rings of $p$-adic Chevalley groups}, Inst. Hautes \'Etudes Sci. Publ. Math. (1965),  5--48.

\bibitem{Kac}
V. Kac, {\em Infinite-dimensional Lie algebras}, Cambridge University Press, Cambridge, third edition, 1990.

\bibitem{Kato}
S. Kato, {\em Loop structure on equivariant $K$-theory of semi-infinite flag manifolds},  Ann. of Math, to appear (\href{https://arxiv.org/abs/1805.01718}{arXiv:1805.01718}).



\bibitem{KLS-Posi}
A. Knutson, T. Lam, and D. Speyer, {\em Positroid varieties: juggling and geometry}, Compos. Math. 149 (2013),  1710--1752. 

\bibitem{KLS-Proj}
A. Knutson, T. Lam, and D. Speyer, {\em Projections of Richardson varieties}, J. Reine Angew. Math. 687 (2014), 133--157.

\bibitem{KJII}
 A. Knutson and P. Zinn-Justin, {\em  Schubert puzzles and integrability II: multiplying motivic Segre class}, 
\href{https://arxiv.org/abs/2102.00563}{arXiv:2102.00563v4}, 2021.





\bibitem{KK86}
B. Kostant and S. Kumar, {\em The nil Hecke ring and cohomology of $G/P$ for a Kac--Moody group G}, Adv.  Math. 62 (1986),  187--237.

\bibitem{Ku02}
S. Kumar, {\em Kac--Moody groups, their flag varieties and representation theory},
Progress in Mathematics, vol. 204, Birkh\"auser Boston, Inc., Boston, MA, 2002, xvi+606 pages.

\bibitem{KRW20}
S. Kumar, R. Rim\'anyi  and A. Weber, {\em  Elliptic classes of Schubert varieties}, Math. Ann. 378 (2020), 703–728.

\bibitem{Lam06}
T. Lam, {\em Affine Stanley symmetric functions}, 
Amer. J. Math. 128 (2006), 1553--1586.



\bibitem{Lam14}
T. Lam, {\em Totally nonnegative Grassmannian and Grassmann polytopes},  Current developments in mathematics 2014, pages 51--152. Int. Press, Somerville, MA, 2016.



\bibitem{LLMSSZ14}
T. Lam, L. Lapointe, J. Morse, A. Schilling, M. Shimozono, and M. Zabrocki, {\em $k$-Schur Functions and Affine Schubert Calculus}, Fields Inst. Monogr., Vol. 33. Springer, New York, 2014.

\bibitem{LLS21}
T. Lam, S. Lee, and M. Shimozono, \emph{Back stable Schubert calculus}, Compos. Math. 157 (2021), 883--962.

\bibitem{LS10}
T. Lam and M. Shimozono, {\em Quantum cohomology of $G/P$ and homology of affine Grassmannian}, Acta. Math. 
204 (2010),  49--90.


\bibitem{Lus98}
G. Lusztig, {\em Total positivity in partial flag manifolds}, Represent. Theory 2 (1998),
70--78.


\bibitem{MacPherson}
R. MacPherson,  {\em Chern classes for singular algebraic varieties}, Ann. of Math.  100 (1974), 423--432.

\bibitem{MO19}
D. Maulik and A. Okounkov, {\em Quantum groups and quantum cohomology},
Ast\'erisque  408 (2019), 1--225.



\bibitem{MNS}
L. Mihalcea, H. Naruse, and C. Su, 
{\em Left Demazure--Lusztig operators on equivariant (quantum) cohomology and K-theory}, Int. Math. Res. Not.(IMRN)  16 (2022), 12096--12147.

\bibitem{MNS22}
L. Mihalcea, H. Naruse, and C. Su, {\em
Hook formulae from Segre--MacPherson classes},
Algebr. Comb., vol. 8 (2025) no. 3, pp. 655-685.


\bibitem{Ohmoto}
T. Ohmoto, {\em Equivariant Chern classes of singular algebraic varieties with group actions}, Math. Proc. Camb. Philos. Soc. 140 (2006), 115--134.

\bibitem{P97}
D. Peterson, {\em Quantum cohomology of $G/P$}, Lecture Notes (M.I.T., Spring 1997).

\bibitem{P06}
A. Postnikov, {\em Total positivity, Grassmannians, and networks}, 
\href{https://arxiv.org/abs/math/0609764}{arXiv:0609764}, 2006.


\bibitem{Rie06}
K. Rietsch, {\em Closure relations for totally nonnegative cells in $G/P$}, Math. Res. Lett. 13 (2006),  775--786.

\bibitem{Sch17}
J. Sch\"urmann, {\em Chern classes and transversality for singular spaces}, In Singularities in Geometry, Topology, Foliations and Dynamics, Trends in Mathematics, pages 207–231. Birkh\"auser, Basel, 2017.

\bibitem{Sch65a} 
M. Schwartz, {\em Classes caract\'eristiques d\'efinies par une stratification
d'une vari\'et\'e analytique complexe, I}, C. R. Acad. Sci. Paris 260 (1965), 3262--3264.

\bibitem{Sch65b} 
M. Schwartz, {\em Classes caract\'eristiques d\'efinies par une stratification
d'une vari\'et\'e analytique complexe, II}, C. R. Acad. Sci. Paris  260 (1965), 3535--3537.

\bibitem{SZ2025}
M. Shimozono and S. Zhang, {\em
Pipe-dreams for affine double Schubert and Grothendieck polynomials}, in preparation.

\bibitem{S23}
D. Speyer, {\em Richardson varieties, projected Richardson varieties and positroid varieties}, \href{https://arxiv.org/abs/2303.04831}{arXiv:2303.04831}, 2023.




\bibitem{Zhu}
X. Zhu, {\em An introduction to affine Grassmannians and the geometric Satake equivalence},
in Geometry of Moduli Spaces and Representation Theory, ed. by R. Bezrukavnikov,
A. Braverman, Z. Yun. IAS/Park City Math. Ser., vol. 24 (American Mathematical
Society, Providence, RI, 2017).

\end{thebibliography}
\end{document}